\documentclass[10pt, a4paper, refqno, twosided]{amsart}
\usepackage{etex}
\usepackage{ifdraft}

\ifdraft{\usepackage[draft]{showkeys}}{\usepackage[final]{showkeys}}
\usepackage[leqno]{amsmath}
\usepackage{amssymb, latexsym, enumerate, MnSymbol, tipa}
\usepackage{amscd,mathrsfs,epic,empheq,float}
\usepackage{bbm}
\usepackage[all]{xy}
\usepackage{wasysym}
\usepackage{pdfsync}
\usepackage{todonotes}
\usepackage[vcentermath]{youngtab}

\usepackage{graphicx,tikz}
\usetikzlibrary{shapes.geometric}
\usetikzlibrary{arrows}

 \allowdisplaybreaks

\numberwithin{equation}{section}

 \newlength{\baseunit}               
 \newcount{\numlines}                
\newcommand\highlightbox[2][]{\tikz[overlay]\node[fill=blue!20,inner sep=1pt, anchor=text, rectangle, rounded corners=1mm,#1]{$\displaystyle #2$};\phantom{#2}}

\makeatletter
\newcommand{\leqnomode}{\tagsleft@true}
\newcommand{\reqnomode}{\tagsleft@false}
\makeatother

\newtheorem{theorem}{Theorem}[section]

\newtheorem{lemma}[theorem]{Lemma}

\newtheorem{prop}[theorem]{Proposition}
\newtheorem{corollary}[theorem]{Corollary}

\theoremstyle{definition}
\newtheorem{definition}[theorem]{Definition}
\newtheorem{define}[theorem]{Definition}
\newtheorem{remark}[theorem]{Remark}

\newcommand{\cO}{{\mathcal O}}

\newcommand{\La}{\Lambda}

\def\down{{\scriptstyle\vee}}
\def\up{{\scriptstyle\wedge}}

\def\down{\vee}
\def\up{\wedge}

\newcommand{\mg}{\mathfrak{g}}
\newcommand{\mh}{\mathfrak{h}}

\newcommand{\pp}{\mathfrak{p}}

\newcommand{\mN}{\mathbb{N}}

\newcommand{\mC}{\mathbb{C}}

\newcommand{\mZ}{\mathbb{Z}}

\newcommand{\SO}{\mathfrak{so}}

\newcommand{\ff}{{\mathbf{f}}}

\newcommand{\la}{\lambda}

\newcommand{\END}{\operatorname{End}}

\newcommand{\op}{\operatorname}

\newcommand{\de}{{\underline\delta}}

\newcommand{\MdV }{M^\mathfrak{p}(\de)\otimes V^{\otimes d}}

\newcommand{\Mde  }{M^\pp(\de)}

\newcommand{\VW}{\bigdoublevee}
\newcommand{\VWd}{{\bigdoublevee}_d}
\newcommand{\VWdcycl}{{\bigdoublevee}_d^{\rm cycl}}

\DeclareMathOperator{\End}{End}

\usepackage[final=true]{hyperref}

\begin{document}
\title[Koszul gradings on Brauer algebras]{Koszul gradings on Brauer algebras}
\author{Michael Ehrig}
\author{Catharina Stroppel}

\address{Department of Mathematics, University of Bonn, 53115 Bonn, Germany}
\email{mehrig@math.uni-bonn.de}
\address{Department of Mathematics, University of Bonn, 53115 Bonn, Germany}
\email{stroppel@math.uni-bonn.de}

\begin{abstract}
We show that the Brauer algebra ${\rm Br}_d(\delta)$ over the complex numbers for an integral parameter $\delta$ can be equipped with a grading, in the case of $\delta \neq 0$ turning it into a graded quasi-hereditary algebra. In which case it is Morita equivalent to a Koszul algebra. This is done by realizing the Brauer algebra as an idempotent truncation of a certain level two VW-algebra for some large positive integral parameter $N$. The parameter $\delta$ appears then in the choice of a cyclotomic quotient. This cyclotomic VW-algebra arises naturally as an endomorphism algebra of a certain projective module in parabolic category $\mathcal{O}$ of type $\rm D$. In particular, the graded decomposition numbers are given by the associated parabolic Kazhdan-Lusztig polynomials.
\end{abstract}

\thanks{M.E. was financed by the DFG Priority program 1388, C.S. thanks the MPI in Bonn for excellent working conditions and financial support.}
\maketitle
\tableofcontents

\section{Introduction}
\renewcommand{\thetheorem}{\Alph{theorem}}
We fix as ground ring the complex numbers $\mathbb{C}$. Given an integer $d\geq 1$ and  $\delta\in\mC$, the  associated {\it Brauer algebra} ${\rm Br}_d(\delta)$ is a diagrammatically defined algebra with basis all Brauer diagrams on $2d$ points, that is all possible matchings of $2d$ points with $d$ strands, such that each point is the endpoint of exactly one strand. In other words, the basis elements correspond precisely to subdivisions of the set of $2d$ points into subsets of order $2$. Here is an example of such a Brauer diagram (with $d=11$):
\begin{equation}
\label{diagram}
\begin{tikzpicture}[thick,>=angle 90]
\begin{scope}[xshift=8cm]
\draw (0,0) -- +(0,1);
\draw (.6,0) -- +(.6,1);
\draw (1.2,0) -- +(-.6,1);
\draw (1.8,0) to [out=90,in=-180] +(.9,.5) to [out=0,in=90] +(.9,-.5);
\draw (1.8,1) to [out=-90,in=-180] +(.6,-.4) to [out=0,in=-90] +(.6,.4);
\draw (2.4,0) -- +(0,1);
\draw (3,0) -- +(.6,1);
\draw (4.2,0) -- +(0,1);
\draw (4.8,0) to [out=90,in=-180] +(.3,.3) to [out=0,in=90] +(.3,-.3);
\draw (4.8,1) to [out=-90,in=-180] +(.3,-.3) to [out=0,in=-90] +(.3,.3);
\draw (6,0) -- +(0,1);
\end{scope}
\end{tikzpicture}
\end{equation}
The multiplication is given on these basis vectors by a simple diagrammatical rule: we fix the positions of the $2d$ points as in the diagram \eqref{diagram} with $d$ points at the bottom and $d$ points at the top. Then the product $bb'$ is equal to $\delta^k (b'\circ b)$, where $b'\circ b$ is the Brauer diagram obtained from the two basis elements $b$ and $b'$ by stacking $b'$ on top of $b$ (identifying the bottom points of $b'$ with the top points of $b$ in the obvious way) and then turning the result into a Brauer diagram by removing all internal circles, and $k$ is the number of such circles removed. For instance:
\begin{equation}
\label{multiplication}
\begin{tikzpicture}[thick,>=angle 90]
\begin{scope}
\draw (0,0) -- +(0,1);
\draw (.6,0) -- +(.6,1);
\draw (1.2,0) to [out=90,in=-180] +(.3,.3) to [out=0,in=90] +(.3,-.3);
\draw (.6,1) to [out=-90,in=-180] +(.6,-.4) to [out=0,in=-90] +(.6,.4);
\node at (2.4,.5) {$\bullet$};
\end{scope}
\begin{scope}[xshift=3cm]
\draw (0,1) to [out=-90,in=-180] +(.3,-.3) to [out=0,in=-90] +(.3,.3);
\draw (1.2,1) to [out=-90,in=-180] +(.3,-.3) to [out=0,in=-90] +(.3,.3);
\draw (0,0) to [out=90,in=-180] +(.6,.4) to [out=0,in=90] +(.6,-.4);
\draw (.6,0) to [out=90,in=-180] +(.6,.4) to [out=0,in=90] +(.6,-.4);
\node at (2.4,.5) {$=$};
\end{scope}
\begin{scope}[xshift=6.5cm]
\node at (-.5,.5) {$\delta$};
\draw (0,1) to [out=-90,in=-180] +(.6,-.4) to [out=0,in=-90] +(.6,.4);
\draw (.6,1) to [out=-90,in=-180] +(.6,-.4) to [out=0,in=-90] +(.6,.4);
\draw (0,0) to [out=90,in=-180] +(.6,.4) to [out=0,in=90] +(.6,-.4);
\draw (.6,0) to [out=90,in=-180] +(.6,.4) to [out=0,in=90] +(.6,-.4);
\end{scope}
\end{tikzpicture}
\end{equation}

Brauer algebras form important examples of cellular diagram algebras in the sense of \cite{GrahamLehrer}. In particular we have {\it cell modules} (or Specht modules) $\Delta(\la)$ indexed by $\la\in\La_d$. Here 
$$\La_d = \bigcup_{m \in \mathbb{Z}_{\geq 0} \cap (d-2\mathbb{Z}_{\geq 0})} {\rm Par}(m),$$
with ${\rm Par}(m)$ denoting the partitions of the integer $m$. We have simple modules $L(\lambda)$ for $\lambda \in \La_d^\delta$, where $\La_d^\delta = \La_d$ in case of $\delta \neq 0$ and $\La_d^\delta = \La_d \setminus {\rm Par}(0)$ in case of $\delta=0$, see \cite{CDM}.

Although the Brauer algebra can be defined for arbitrary $\delta \in \mathbb{C}$ it turns out that it is always semi-simple for $\delta \notin \mathbb{Z}$, see \cite{Rui}. For our purposes these cases are trivial, hence we will always assume $\delta \in \mathbb{Z}$.

Brauer algebras were originally introduced by Brauer \cite{Brauer} in the context of classical invariant theory as centralizer algebras of tensor products of the natural representation of orthogonal and symplectic Lie algebras. More precisely, assuming $d<n$ there is a canonical isomorphism of algebras
\begin{eqnarray}
\label{centralizer}
\End_\mg(V^{\otimes d})&\cong& {\rm Br}_d(N)
\end{eqnarray}
where $\mg$ is an orthogonal or symplectic Lie algebra of rank $n$ with vector representation $V$ of dimension $N$ in the orthogonal case and dimension $-N$ in the symplectic case, see e.g. \cite{CP} or \cite{GW} for details. 

As an algebra, the Brauer algebra is generated by the following elements $t_{i}$, $g_i$ 
\begin{equation}
\label{diag}
\begin{tikzpicture}[scale=0.7,thick,>=angle 90]
\node at (0,.5) {$t_i$};
\draw (.6,0) -- +(0,1);
\draw [dotted] (1,.5) -- +(1,0);
\draw (2.4,0) -- +(0,1);
\draw (3,0) -- +(.6,1) node[above] {\tiny i+1};
\draw (3.6,0) -- +(-.6,1) node[above] {\tiny i};
\draw (4.2,0) -- +(0,1);
\draw [dotted] (4.6,.5) -- +(1,0);
\draw (6.2,0) -- +(0,1);

\begin{scope}[xshift=8cm]
\node at (0,.5) {$g_i$};
\draw (.6,0) -- +(0,1);
\draw [dotted] (1,.5) -- +(1,0);
\draw (2.4,0) -- +(0,1);
\draw (3,0) to [out=90,in=-180] +(.3,.3) to [out=0,in=90] +(.3,-.3);
\draw (3,1) node[above] {\tiny i} to [out=-90,in=-180] +(.3,-.3) to [out=0,in=-90] +(.3,.3) node[above] {\tiny i+1};
\draw (4.2,0) -- +(0,1);
\draw [dotted] (4.6,.5) -- +(1,0);
\draw (6.2,0) -- +(0,1);
\end{scope}
\end{tikzpicture}
\end{equation}
for $1\leq i < d$, and $t_i$ acts on $V^{\otimes d}$ in \eqref{centralizer} by permuting the $i$th and $(i+1)$st tensor factor, and the element $g_{i}$ acts by applying to the $i$th and $(i+1)$st factor the evaluation morphism $V\otimes V=V^*\otimes V\rightarrow \mathbb{C}$ followed by its adjoint.  

The realization \eqref{centralizer} as centralizers includes the cases ${\rm Br}_d(\delta)$ for $\delta\in\mathbb{Z}$ integral and $\delta$ large enough in comparison to $d$. Hence the Brauer algebra is semisimple in these cases. In fact it was shown by Rui, see \cite{Rui}, that  ${\rm Br}_d(\delta)$ is semisimple except for $\delta$ integral of small absolute value, see also \cite{ES3} for a precise statement and references therein for proofs. For arbitrary $\delta\in\mathbb{Z}$ and $d\geq 1$ the Brauer algebras still appear as centralizers of the form \eqref{centralizer} if we replace $\mathfrak{g}$ by an orthosymplectic Lie superalgebra such that its vector representation has super dimension $k|2n$ with $\delta=k-2n$, see \cite{ESSchurWeyl}. \\

Whereas the semisimple cases were studied in detail in many papers, including for instance the semiorthogonal form in \cite{Nazarov}, the non-semisimple cases are still not well understood. In the pioneering work of Cox, De Visscher and Martin, \cite{CDM}, it was observed that the multiplicity $[\Delta(\la)\,:\,L(\mu)]$ how often a simple module $L(\mu)$ indexed by $\mu$ (that is the simple quotient of $\Delta(\mu)$) appears in a Jordan-H\"older series of the cell module $\Delta(\la)$ is given by certain parabolic Kazhdan-Lusztig polynomial $n_{\la,\mu}$ of type ${\rm D}$ with parabolic of type ${\rm A}$, \cite{Boe}, \cite{LS},  i.e. 
\begin{eqnarray}
\label{KL}
[\Delta(\la)\, :\, L(\mu)]&=&n_{\la,\mu}(1).
\end{eqnarray}
This result connects the combinatorics of Brauer algebras with Kazhdan-Lusztig combinatorics of type ${\rm D}$ Lie algebras, i.e. multiplicities of simple (possibly infinite) highest weight modules appearing in a parabolic Verma module. Here two obvious questions arise: Is there an interpretation of the variable $q$ in the Kazhdan-Lusztig polynomial $n_{\la,\mu}(q)\in\mZ[q]$? 
Is there an equivalence of categories between modules over the Brauer algebra  ${\rm Br}_d(\delta)$  for integral $\delta$ and some subcategory of the Bernstein-Gelfand-Gelfand (parabolic) category $\cO$ for type ${\rm D}$ explaining the mysterious match in the combinatorics? In this paper we will give an answer to both questions.

Let us  explain the results in more detail. Given a finite dimensional algebra $A$ we denote by $A-\operatorname{mod}$ its category of finite dimensional modules. If the algebra $A$ is $\mZ$-graded we denote by $A-\operatorname{gmod}$ its category of finite dimensional graded modules with degree preserving morphisms and by $F:A-\operatorname{gmod}\rightarrow A-\operatorname{mod}$ the grading forgetting functor. For $i\in\mathbb{Z}$ let $\langle i\rangle: M\mapsto M\langle i\rangle$ denote the autoequivalence of $A-\operatorname{gmod}$ which shifts the grading by $i$, i.e. $F(M\langle i\rangle )=FM$ and $(M\langle i\rangle )_j=M_{j-i}$ for any $M\in A-\operatorname{gmod}$. As an application of our main theorem below we obtain the following refinement of \eqref{KL}, summing up the results obtained in Section \ref{sec:consequences}:

\begin{theorem}
\label{first}
Let $\delta\in\mathbb{Z}$. The Brauer algebra ${\rm Br}_d(\delta)$ can be equipped with a $\mathbb{Z}$-grading turning it into a $\mathbb{Z}$-graded algebra ${\rm Br}^{\rm gr}_d(\delta)$. Moreover, it satisfies the following: 
\begin{enumerate}[{\rm 1.)}]
\item ${\rm Br}^{\rm gr}_d(\delta)$ is Morita equivalent to a Koszul algebra if and only if $\delta\not=0$ or $\delta=0$ and $d$.
\item ${\rm Br}^{\rm gr}_d(\delta)$ is graded cellular. 
\item ${\rm Br}^{\rm gr}_d(\delta)$ is graded quasi-hereditary if and only if $\delta\not=0$ or $\delta=0$ and $d$ odd. Moreover, in the quasi-hereditary case
\begin{enumerate}[{\rm a.)}]
\item We have distinguished graded lifts of standard modules and simple modules in the following sense: For the cell module $\Delta(\lambda)$ of ${\rm Br}_d(\delta)$, $\lambda \in \La_d$, there exists a module $\widehat{\Delta}(\lambda)$ for  ${\rm Br}^{\rm gr}_d(\delta)$ such that $F\widehat{\Delta}(\lambda) = \Delta(\lambda)$. For a simple module $L(\mu)$ of ${\rm Br}_d(\delta)$, for $\mu \in \La_d^\delta$, there exists a module $\widehat{L}(\mu)$ for  ${\rm Br}^{\rm gr}_d(\delta)$ such that $F\widehat{L}(\mu) = L(\mu)$. Furthermore $\widehat{L}(\mu)$ is the simple quotient of $\widehat{\Delta}(\mu)$ concentrated in degree $0$, making the choice of these lifts unique.
\item The $\widehat{\Delta}(\la)$ form the graded standard modules.
\end{enumerate}
\item The modules $\widehat{\Delta}(\la)$ have a Jordan-H\"older series in ${\rm Br}^{\rm gr}_d(\delta)-\operatorname{gmod}$ with multiplicities  given by 
\begin{eqnarray*}
\left[\widehat{\Delta}(\la)\, :\, \widehat{L}(\mu)<i>\right]&=&n_{\la,\mu,i},
\end{eqnarray*}
where $n_{\la,\mu}(q)=\sum_{i\geq 0}n_{\la,\mu,i} q^i$.
\end{enumerate}
\end{theorem}
For instance, ${\rm Br}^{\rm gr}_2(\delta)$ is isomorphic to the algebra $\mathbb{C}\oplus\mathbb{C}\oplus\mathbb{C}$ in case $\delta\not=0$ whereas it is isomorphic to $\mathbb{C}\oplus \mathbb{C}[x]/(x^2)$ with $x$ in degree $2$, see Section~\ref{sec:example}.\\

The above result is based on our main theorem which realizes ${\rm Br}_d(\delta)$ for integral $\delta$ as an {\it idempotent truncation} of a level $2$ cyclotomic quotient of a VW-algebra $\VW_{d}(\Xi)$, see Definition \ref{defalphabeta} for the exact parameter set $\Xi$. Here, the {\it VW-algebra} $\VW_{d}(\Xi)$ is as a vector space isomorphic  to  ${\rm Br}_d(N)\otimes \mC[y_1,\ldots y_d]$ with both factors being in fact subalgebras. The defining relations imply that there is a unique surjective homomomorphism of algebras
\begin{eqnarray}
\VWd(\Xi)&\longrightarrow&{\rm Br}_d(N),
\end{eqnarray}
 which extends the identity on ${\rm Br}_d(N)$ and sends $y_1$ to $0$. The polynomial generators $y_k$ are then sent to the famous {\it Jucys-Murpy elements} $\xi_k$ in the Brauer algebra, see Proposition \ref{jucysmurphy} for a definition. These elements form a commutative subalgebra which plays an important role in the theory of semiorthogonal forms for the Brauer algebras. In this way, the Brauer algebra ${\rm Br}_d(N)$ can be realized naturally as a level 1 cyclotomic quotient of $\VW_{d}(\Xi)$.
 
 The connection to Lie theory however is based on a more subtle realization of the Brauer algebra as follows: Let $n\in\mathbb{Z}$ be large (say $N=2n\geq 2d$) and consider the the type ${\rm D}_n$ Lie algebra $\SO(N)$ of rank $n$ with its vector representation $V$. 
 Let $\varpi_0$ be the fundamental weight corresponding to a spin representation (that is to one of the fork ends in the Dynkin diagram) and let $\pp\subset \SO(N)$ be the (type $A$) maximal parabolic corresponding to the simple roots orthogonal to $\varpi_0$. For any fixed $\delta\in\mathbb{Z}$ let $M^\pp(\delta)$ be the associated parabolic Verma module with highest weight $\delta\omega_0$, see \cite{Humphreys}. Then \cite[Theorem 3.1]{ES3} gives a natural isomorphism of algebras
 \begin{eqnarray}
\END_{\SO(N)}(\MdV)^{\op{opp}}\cong \VWdcycl 
\end{eqnarray}
where $\VWdcycl = \VWd(\Xi)/(y_1-\alpha)(y_1-\beta)$ for $\alpha=\frac{1}{2}(1-\delta)$ and $\beta=\frac{1}{2}(\delta+N-1)$. The finite dimensional algebra $\VWdcycl$ decomposes into simultaneous generalized eigenspaces with respect to $ \mC[y_1,\ldots y_d]$. Let $\ff$ be the idempotent of $\VWdcycl$ which projects onto all common generalized eigenspaces with {\it small} eigenvalues $c_j$ with respect to $y_j$, i.e. where $c_j$ satisfy $|c_j|<\beta$ for $1\leq j\leq d$. 

Our main result, Theorem \ref{thm:main}, is then the following:

\begin{theorem}
For any fixed $\delta\in\mathbb{Z}$, there is an isomorphism of algebras 
\begin{eqnarray*}
\Phi_\delta:&& {\rm Br}_d(\delta)\longrightarrow \ff \VWdcycl \ff
\end{eqnarray*}
given on the standard generators of the Brauer algebra by 
\begin{eqnarray}
\label{difficult1intro}
t_k \quad\longmapsto\quad -Q_k s_k Q_k + \frac{1}{b_k}\ff,&\text{and}& g_k \quad\longmapsto\quad Q_k e_k Q_k.
\end{eqnarray}
\end{theorem}
for certain elements $Q_k$ and $b_k$, defined in \eqref{19} and \eqref{20}, which can be expressed in term of the polynomial generators $y_j$, $1\leq j \leq d$ and $\beta$.  

 Note that the idempotent truncation is independent of $N$ (as long $N$ is large enough), but the right hand side as well as the map do depend on $N$. In particular, the parameter $N$ on the right hand side changes into the parameter $\delta$ on the left hand side. Under the isomorphism, the Jucys-Murphy elements $\xi_k$ of the Brauer algebra are sent to $-y_k \ff$ inside $\ff \VWdcycl \ff$, see Proposition \ref{jucysmurphy}. In particular, generalized eigenspaces for Jucys-Murphy elements coincide with generalized eigenspaces of the polynomial generators $y_k$.
 
 By general theory on category $\mathcal{O}$, \cite{BGS}, \cite{Backelin},  the algebra $\VWdcycl$ can be equipped with a positive $\mathbb{Z}$-grading, see \cite[Theorem 3.1]{ES3}. Since the idempotent truncation $\mathbf{f}$ corresponds to successive projections onto blocks, see \cite[Section 4.1]{ES3}, $B_d(\delta)$ inherits a grading which is then the grading in Theorem \ref{first}. In contrast to a general block in category $\mathcal{O}$, the grading can be made totally explicit in our case using the theory of generalized Khovanov algebras of type $D$, \cite{ES3}. Note that the theorem implies that all of the combinatorics developed in \cite{ES3} for $\ff \VWdcycl \ff$ are now applicable to the Brauer algebra. 
 
 As an application  of our result note that understanding the degree of non-semisimplicity for Brauer algebras and decomposition numbers in the non-semisimple cases gives some first insight into the structure of the tensor product of the natural module for the orthosymplectic Lie superalgebra via the result from \cite{ESSchurWeyl}, or \cite{LZ}.\\
 
 The idea and difficulty behind the formulas \eqref{difficult1intro} stems from the fact that we connect directly the semiorthogonal form for $\VWdcycl$ from \cite{AMR} and for  ${\rm Br}_d(\delta)$ from \cite{Nazarov} by realizing the latter as obtained from the first via a naive idempotent truncation to small eigenvalues corrected with some extra terms encoded in the rather complicated elements $Q_k$ and $b_k$. The correction terms seriously depend on the generalized weight spaces. In the semisimple case, i.e if all generalized eigenvectors are actually proper eigenvectors, the formulas simplify drastically, since the operators act on an eigenspace just by a certain number which can be expressed combinatorially in terms of contents of tableaux. The point here is that our formulas also work in the non-semisimple cases. Then the definitions of the correction terms involves (inverses) of square roots. It is a nontrivial result, that the operators are well-defined and that the images of the generators of the Brauer algebra satisfy the Brauer algebra relations. 
 
 We should mention that a similar result for walled Brauer algebras was obtained in \cite{SS2}. The two results are independent and, as far as we can see,  neither of the two implies the other. In fact, the result \cite[Lemma 8.1]{SS2} seriously simplifies the setup treated there, but doesn't hold in the Brauer algebra setting. As a result, the Brauer algebra requires a very different treatment. 
 
 Another approach defining gradings on Brauer algebras arising from semiorthogonal forms was taken independently in \cite{Li} and resulted in a KLR-type presentation of Brauer algebras. One can show that the two algebras are actually isomorphic as graded algebras, see \cite{LiS}. In particular, parabolic category $\mathcal{O}$ of type $D$ should provide a general framework for higher level cyclotomic quotients of VW-algebras provide some KLR-type presentations similar to the beautiful results in type A, see e.g. \cite{BK}, \cite{BS3}, \cite{HM}. 

{\bf Acknowledgment} We like to thank Jonathan Brundan and Antonio Sartori for many helpful discussions. 

\section{Brauer algebra and VW algebras}
\renewcommand{\thetheorem}{\arabic{section}.\arabic{theorem}}
We start with the definition of the Brauer algebra in terms of generators and relations.  Then we recall the definition of its degenerate affine analogue, the so-called VW-algebra with its cyclotomic quotients. 
By an algebra we always mean an associative  unitary algebra with unit $1$.

\begin{definition}
Let $d \in \mN$ and $\delta \in \mathbb{C}$. The {\it Brauer algebra} ${\rm Br}_d(\delta)$ is the associative $\mathbb{C}$-algebra generated by elements $t_i$, $g_i$, $1\leq i\leq d-1$ subject to the relations
\begin{equation*}
\begin{array}{lll}
t_i^2=1, & t_i t_j= t_j t_i \text{ for } |i-j|>1, & t_i t_{i+1} t_i = t_{i+1} t_i t_{i+1}, \\
g_i^2=\delta g_i, & g_i g_j= g_j g_i \text{ for } |i-j|>1, & g_i g_{i+1} g_i = g_i, \\
t_i g_i= g_i = g_i t_i, & g_i t_j= t_j g_i \text{ for } |i-j|>1, &g_{i+1} g_{i} g_{i+1}= g_{i+1}, \\
t_i g_{i+1} g_i = t_{i+1} g_i, & t_{i+1} g_{i} g_{i+1} = t_{i} g_{i+1}.&
\end{array}
\end{equation*}
whenever the terms in the expressions are defined. 
\end{definition}


\begin{remark} \label{rem:different_rel}
Since $g_i g_{j} g_i = g_i t_i g_{j} g_i = g_i t_j g_i$ one can replace the relation $g_i t_{j} g_i = g_i$ for $|i-j| = 1$ by the relation $g_{i} g_{j} g_{i}= g_{i}$.
\end{remark}

\begin{define}
\label{def:VW}
Let $d \in \mN$ and $\Xi = (\omega_i)_{i \in \mathbb{N}}$ with $\omega_i \in
\mathbb{C}$ for all $i$. Then the associated \emph{VW-algebra} $\VWd(\Xi)$ is the algebra generated by
\begin{eqnarray}
s_i, e_i, y_j&&1 \leq i \leq d-1, 1 \leq i \leq d, k \in \mN,
\end{eqnarray}
subject to the following relations (for $1
\leq a,b \leq d-1$, $1 \leq c < d-1$, and $1 \leq i,j \leq d$):
\begin{enumerate}[(VW.1)]
\item $s_a^2 = 1$ \label{1}
\item
\begin{enumerate}
\item $s_as_b = s_bs_a$ for $\mid a-b \mid > 1$ 	 \label{2a}
\item $s_c s_{c+1} s_c = s_{c+1} s_c s_{c+1}$ 	 \label{2b}
\item $s_ay_i = y_is_a$ for $i \not\in \{a,a+1\}$ \label{2c}
\end{enumerate}
\item $e_a^2 = \omega_0 e_a$ \label{3}
\item $e_1y_1^ke_1 = \omega_k e_1$ for $k \in \mN$ \label{4}
\item
\begin{enumerate}
\item $s_ae_b = e_bs_a$ and $e_ae_b = e_be_a$ for $\mid a-b \mid > 1$
\label{5a}
\item $e_ay_i = y_ie_a$ for $i \not\in \{a,a+1\}$ \label{5b}
\item $y_iy_j = y_jy_i$ \label{5c}
\end{enumerate}
\item
\begin{enumerate}
\item $e_as_a = e_a = s_ae_a$ 	\label{6a}
\item $s_ce_{c+1}e_c = s_{c+1}e_c$ and $e_ce_{c+1}s_c = e_cs_{c+1}$
\label{6b}
\item $e_{c+1}e_cs_{c+1} = e_{c+1}s_c$ and $s_{c+1}e_ce_{c+1} =
s_ce_{c+1}$ \label{6c}
\item $e_{c+1}e_ce_{c+1} = e_{c+1}$ and $e_ce_{c+1}e_c = e_c$ 	\label{6d}
\end{enumerate}
\item $s_ay_a - y_{a+1}s_a = e_a - 1$ and $y_as_a - s_ay_{a+1} = e_a - 1$
\label{7}
\item
\begin{enumerate}
\item $e_a(y_a+y_{a+1}) = 0$ \label{8a}
\item $(y_a+y_{a+1})e_a = 0$ \label{8b}
\end{enumerate}
\end{enumerate}
\end{define}

\begin{remark}
The VW-algebra $\VWd$ is a degeneration of the affine BMW-algebra~\footnote{as every car lover can probably imagine easily ...}, \cite{DRV}, hence plays the analogue role for the Brauer algebra as the degenerate affine Hecke algebra for the symmetric group.  It was introduced originally by Nazarov in \cite{Nazarov} under the name generalized Wenzl-algebra \footnote{which translated to German is {\it Verallgemeinerte Wenzl Algebra}, abbreviated as $\bigdoublevee$. It is also sometimes called {\it Nazarov-Wenzl algebra} in the literature. Hence $\bigdoublevee$ can be viewed as composed of the letters $N$, $W$ and $V$ as well.}.
\end{remark}


Finally we introduce, following \cite{AMR},  the cyclotomic quotients of $\VWd$ of level $\ell$:

\begin{definition}
Given $\mathbf{u}=(u_1, u_2,\ldots, u_\ell)\in \mathbb{C}^\ell$ we denote by $\VWd(\Xi;{\bf u})$ the quotient
\begin{eqnarray}
\VWd(\Xi,{\bf u})&=&\VWd(\Xi)/\prod_{i=1}^\ell(y_1-u_i)
\end{eqnarray}
and call it the \emph{cyclotomic VW-algebra of level $\ell$} with parameters $\bf{u}$.
\end{definition}

\begin{remark}
As explained in \cite{AMR}, the tuple $\Xi$ must satisfy some admissibility condition for the algebra $\VWd(\Xi)$ to have a nice basis. Furthermore, the $\Xi$ must satisfy some $\mathbf{u}$-admissibility condition for this basis to be compatible with the quotient, see \cite[Theorem A, Prop. 2.15]{AMR}.
\end{remark}


Inside the VW-algebra $\VWd(\Xi)$, the elements $\{y_k | 1 \leq k \leq d \}$ generate a free commutative subalgebra, hence we can consider the simultaneous generalized eigenspace decompositions for these elements. Any finite dimensional $\VWd(\Xi)$-module $M$ has a  decomposition
\begin{eqnarray}
\label{eigspaces}
M &=& \bigoplus_{\mathbf{i} \in \mathbb{C}^d} M_{\mathbf{i}},
\end{eqnarray}
where $M_{\textbf{i}}$ is the generalized eigenspace with eigenvalue $\mathbf{i}$, i.e., $(y_k - \mathbf{i}_k)^N M_{\mathbf{i}} = 0$ for $N \gg 0$ sufficiently large. We first describe how the generators $e_k$ and $s_k$ interact with this eigenspace decomposition.

\begin{lemma} \label{lem:eigenvalue_and_e}
For all $1 \leq k < d$ the following holds
$$ e_k M_{\mathbf{i}} \subset \left\lbrace 
\begin{array}{ll}
\{0\} & \text{if } \mathbf{i}_k + \mathbf{i}_{k+1} \neq 0, \\
\bigoplus_{\mathbf{i}' \in {\rm I}} M_{\mathbf{i}'} & \text{if } \mathbf{i}_k + \mathbf{i}_{k+1} = 0,
\end{array} \right.$$
where ${\rm I} = \{\mathbf{i}' \in \mathbb{C}^d | \mathbf{i}_j' = \mathbf{i}_j \text{ for } j \neq k,k+1 \text{ and } \mathbf{i}_k' + \mathbf{i}_{k+1}' = 0\}$.
\end{lemma}
\begin{proof}
Assume first that $a := \mathbf{i}_k+\mathbf{i}_{k+1} \neq 0$. Then the endomorphism induced by $(y_k+y_{k+1} - a)$ is nilpotent on $M_\mathbf{i}$ and hence $y_k+y_{k+1}$ induces an automorphism of $M_\mathbf{i}$. Hence $e_k M_\mathbf{i} = e_k (y_k + y_{k+1})M_\mathbf{i} = \{0\}$, where for the last equality Relation~(VW.\ref{8a}) was used.

Now assume $\mathbf{i}_k+\mathbf{i}_{k+1} = 0$. Since $(y_k + y_{k+1})e_k = 0$ by Relation~(VW.\ref{8b}) we know that on the image of $e_k$ the endomorphism induced by $y_k+y_{k+1}$ has eigenvalue $0$, hence $y_k$ and $y_{k+1}$ have eigenvalues that add up to $0$.
\end{proof}

The situation for $s_k$ is more complicated and we need the following preparation:

\begin{lemma} \label{lem:eigenvalue_and_s}
Let $\psi_k = s_k(y_k - y_{k+1}) +1$, then for all $1 \leq k \leq d-1$ it holds
$$ \psi_k M_{\mathbf{i}} \subset \left\lbrace 
\begin{array}{ll}
M_{s_k \mathbf{i}} & \text{if } \mathbf{i}_k + \mathbf{i}_{k+1} \neq 0, \\
\bigoplus_{\mathbf{i}' \in {\rm I}} M_{\mathbf{i}'} & \text{if } \mathbf{i}_k + \mathbf{i}_{k+1} = 0,
\end{array} \right.$$
In case $\mathbf{i}_k+\mathbf{i}_{k+1}\neq 0$ and additionally $|\mathbf{i}_k-\mathbf{i}_{k+1}| \neq 1$, the map $\psi_k$ defines an isomorphism of vector spaces $M_\mathbf{i} \cong M_{s_k(\mathbf{i})}$.
\end{lemma}
\begin{proof}
From Relation~(VW.\ref{7}) we obtain  $y_k \psi_k = \psi_k y_{k+1} + e_k (y_k-y_{k+1})$ and by Relation~(VW.\ref{2c}) then $y_j \psi_k = \psi_k y_j$ for $|j-k|>1$.

Assume $\mathbf{i}_k + \mathbf{i}_{k+1} \neq 0$ and let $m \in M_{\mathbf{i}}$. Then by Lemma \ref{lem:eigenvalue_and_e} we have $y_k \psi_k m = \psi_k y_{k+1} m$ and $y_{k+1} \psi_k = \psi_k y_{k} m$, hence $(y_i - \mathbf{i}_{s_k(i)})^N \psi_k m = \psi_k (y_{s_k(i)} - \mathbf{i}_{s_k(i)})^N m$ for all $N$ and all $i$. Thus $\psi_k M_\mathbf{i} \subset M_{s_k(\mathbf{i})}$.

Relation~(VW.\ref{7}) implies $(y_k+y_{k+1})\psi_k = \psi_k(y_k+y_{k+1})$. Assuming $i_k + i_{k+1} = 0$ we have $(y_k+y_{k+1})^N M_\mathbf{i} = \{0\}$ and hence $ (y_k+y_{k+1})^N \psi_k M_\mathbf{i} = \psi_k (y_k+y_{k+1})^N M_\mathbf{i} = \{0\}$. Thus, on the image of $\psi_k$, the endomorphism induced by $y_k+y_{k+1}$ has eigenvalue $0$ and therefore $y_k$ and $y_{k+1}$ have eigenvalues adding up to $0$.

Assuming now $\mathbf{i}_k + \mathbf{i}_{k+1} \neq 0$ and furthermore $|\mathbf{i}_k-\mathbf{i}_{k+1}| \neq 1$, then thanks to Relation~(VW.\ref{7}) and Lemma \ref{lem:eigenvalue_and_e} we have  $\psi_k^2=-(y_k-y_{k+1})^2+1$ as an endomorphisms of $M_\mathbf{i}$. Setting $c=\mathbf{i}_k - \mathbf{i}_{k+1}$ it follows $\left( (y_k-y_{k+1}) - c \right)^N M_\mathbf{i} = \{ 0 \}$ for $N \geq 0$ by definition. In particular, as endomorphisms of $M_\mathbf{i}$, this implies 
\begin{eqnarray*}
\psi_k^2 &=& 1-\left((y_k-y_{k+1})-c+c\right)^2 \\
&=&1- c^2 - \left((y_k-y_{k+1})-c \right)^2 - c (y_k-y_{k+1}-c) = 1-c^2-z
\end{eqnarray*}
for some nilpotent endomorphism $z$. Since $c^2 \neq 1$ by assumption, $\psi_k^2$ is invertible and therefore also $\psi_k$. Note that the concrete form of the inverse depends on $\mathbf{i}$.
\end{proof}

\begin{corollary} \label{cor:permute_eigenvalues}
Assume $M_\mathbf{i} \neq \{0\}$ for some $\mathbf{i} = (\mathbf{i}_1,\dots, \mathbf{i}_{d-1}, \mathbf{i}_d)$ such that $\mathbf{i}_k + \mathbf{i}_d \neq 0$ and $|\mathbf{i}_k - \mathbf{i}_d| \neq 1$ for all $k < d$. Let $\mathbf{i}'=(\mathbf{i}_d, \mathbf{i}_1,\dots, \mathbf{i}_{d-1})$, then it holds
$$(y_d - \mathbf{i}_d)^N M_\mathbf{i}=\{0\} \Longleftrightarrow (y_1-\mathbf{i}_d)^N M_{\mathbf{i}'}=\{0\}$$
for all positive integers $N$.
\end{corollary}
\begin{proof}
By Lemma \ref{lem:eigenvalue_and_s} the element $\psi:=\psi_1 \cdots \psi_{d-1}$ as an isomorphism between $M_{\mathbf{i}}$ and $M_{\mathbf{i}'}$ intertwining the actions of $y_1$ and $y_d$. i.e. $\psi(y_1m)=y_d\psi(m)$ for any 
$m\in M_{\mathbf{i}}$.
\end{proof}

\begin{corollary} \label{cor:eigenvalue_and_s}
Assume that $\mathbf{i}_k \neq \mathbf{i}_{k+1}$, then
$$s_k M_{\mathbf{i}} \subset \left\lbrace 
\begin{array}{ll}
M_{s_k \mathbf{i}} \oplus M_{\mathbf{i}} & \text{if } \mathbf{i}_k + \mathbf{i}_{k+1} \neq 0, \\
\bigoplus_{\mathbf{i}' \in {\rm I}} M_{\mathbf{i}'} & \text{if } \mathbf{i}_k + \mathbf{i}_{k+1} = 0.
\end{array} \right.$$
\end{corollary} 
\begin{proof}
Under the assumption that $\mathbf{i}_k \neq \mathbf{i}_{k+1}$ we have that $y_k-y_{k+1}$ is an automorphism of $M_\mathbf{i}$ and the statement follows then directly from Lemma \ref{lem:eigenvalue_and_s}.
\end{proof}

\section{Cyclotomic quotients and category $\mathcal{O}$}
\label{sec:cycl_and_cat_O}
Fix $\delta \in \mZ$. Let $\mathfrak{g}=\mathfrak{so}(2n)$ be the complex special orthogonal Lie algebra corresponding to the Dynkin diagram $\Gamma$ of type ${\rm D}_n$ and fix a triangular decomposition $\mathfrak{g}= \mathfrak{n}^- \oplus \mathfrak{h} \oplus \mathfrak{n}^+$. Fix $\mathfrak{l}$, a Levi subalgebra obtained from an embedding of the type ${\rm A}_{n-1}$ Dynkin diagram into $\Gamma$ and denote by $\mathfrak{p} = \mathfrak{l} \oplus \mathfrak{n}^+$ the corresponding parabolic subalgebra. Denote by $\varepsilon_1,\ldots,\varepsilon_n$ the standard basis of $\mathfrak{h}^*$.

By $\mathcal{O}^\mathfrak{p}(n) = \mathcal{O}^\mathfrak{p}_{\rm int}(\mathfrak{so}(2n))$ we denote the integral parabolic BGG category $\mathcal{O}$, i.e., the full subcategory of $\mathcal{U}(\mathfrak{g})$-modules consisting of finitely generated $\mathcal{U}(\mathfrak{g})$-modules, semisimple over $\mathfrak{h}$ with integral weights, and locally finite for $\mathfrak{p}$, see \cite[Chapter 9]{Humphreys}. Let
\begin{equation}
X_n^\mathfrak{p} = \{\la\in\mh^*\text{ integral}\mid\la+\rho=\sum_{i=1}^n\la_i\epsilon_i \text{ where } \la_1< \la_2<\cdots< \lambda_n\},
\label{eqn:DefLa}
\end{equation}

where $\rho$ denotes the half-sum of the positive roots, $\rho=\sum_{i=1}^n(i-1)\epsilon_i$. Then $X_n^\mathfrak{p}$ is precisely the set of highest weights of simple objects in $\mathcal{O}^\mathfrak{p}(n)$, see \cite{Humphreys}. For an element $\lambda \in X_n^\mathfrak{p}$ we denote by $+M^\pp(\la)$ the parabolic Verma module with highest weight $\lambda$. Note that a weight $\la=(\la_1,\la_2,\cdots,\la_n)$ written in the $\epsilon$-basis is integral, if either $\la_i\in\mathbb{Z}$, i.e. $2\la_i$ is even for all $i$, or $\la_i\in\frac{1}{2}+\mathbb{\mZ}$, i.e. $2\la_i$ is odd for all $i$. 

As in  \cite{ES3}, a crucial player in the following will be the parabolic Verma module $M^\mathfrak{p}(\underline{\delta})$ of highest weight
\begin{eqnarray} \label{eqn:delta}
\underline{\delta}& =&\frac{\delta}{2}(\varepsilon_1 + \ldots + \varepsilon_n),
\end{eqnarray}
i.e., a multiple of the fundamental weight $\frac{1}{2}(\varepsilon_1 + \ldots + \varepsilon_n)$. With an appropriate choice of parameters $\Xi_\delta$, there is , see \cite{ES3}, a natural (right) action of $\VWd(\Xi_\delta)$ on $\Mde \otimes V^{\otimes d}$ by $\mathfrak{g}$-endomorphisms. Hence we have an algebra homomorphism $\VWd(\Xi_\delta)\rightarrow \END_\mg(\MdV)^{\op{opp}}$.  The parameter set $\Xi_\delta = (\omega_a)_{a \geq 0}$ appear as part of the following definition':

\begin{definition}
\label{defalphabeta}
For $N = 2n$, we define the \emph{cyclotomic parameters} as follows
\begin{eqnarray}
&\omega_0=N, \quad \omega_1=N\frac{N-1}{2}, \quad \omega_a=(\alpha+\beta)\omega_{a-1}-\alpha\beta \omega_{a-2} \quad \text{for } a \geq 2,&\label{NN}\\
&\text{where we set}&\nonumber\\
&\alpha=\frac{1}{2}(1-\delta),\quad \quad\quad \beta=\frac{1}{2}(\delta+N-1).\label{ab}&
\end{eqnarray}
\end{definition}

Then the following important result holds:

\begin{theorem}[\cite{ES3}]
\label{iso}
If $n \geq 2d$, then the $\VWd(\Xi_\delta)$-action from above  induces an isomorphism of algebras
\begin{eqnarray}
\label{Psi}
\Psi(\de):\quad\VWd(\Xi_\delta;\alpha,\beta)&\longrightarrow&\END_\mg(\MdV)^{\op{opp}}.
\end{eqnarray}
\end{theorem}

In the following we abbreviate $\VWdcycl=\VWd(\Xi_\delta;\alpha,\beta)$. \\

Note that via \eqref{Psi}, the space $\MdV$ becomes a module for $\VWd(\Xi_\delta;\alpha,\beta)$ with the action preserving the finite dimensional $\mg$-weight spaces, hence we have a simultaneous generalized eigenspace decomposition \eqref{eigspaces}. We describe now this decomposition Lie theoretically and then combinatorially using the notion of a up-down bitableaux and bipartitions. We start with the following well-known fact:

\begin{lemma}
\label{flags}
Let $\mu \in \La$. Then $M^\mathfrak{p}(\mu)\otimes V$ has a filtration (called Verma flag) 
with sections isomorphic to precisely the 
$M^\mathfrak{p}(\mu\pm\epsilon_j)$ for all
$j=1,\dots n$ such that
$\mu\pm\epsilon_j \in \La$. The sections are pairwise not isomorphic.  
\end{lemma}

\begin{proof}
This is a standard consequence of the tensor identity;
see e.g. \cite[Theorem 3.6]{Humphreys}, noting that $V$ has precisely the weights $\pm\epsilon_j$ for $1\leq j\leq n$.
\end{proof}

In particular, $M^\mathfrak{p}(\underline{\delta})\otimes V^{\otimes{d}}$ has a Verma flag with sections isomorphic to precisely the 
$M^\mathfrak{p}(\la)$, where  $\la-\de=\sum_{j=i}^d a_i$, where $a_i\in\{\pm\epsilon_j\mid1\leq j\leq n\}$.  Given such a weight $\la$ we can write $\la-\de=\sum_{i=1}^n m_i\epsilon_i$ with $m_i\leq 0$ for $1\leq i\leq s$ and $m_i>0$ for $i>s$ for some (uniquely defined) $s$.  Then we assign to $\la$ the bipartition  $\varphi(\la)=(\la^{(1)},\la^{(2)})$ with $\la^{(1)}_i=|m_i|$ and $\la^{(2)}_i=m_{s+i}$. This will be seen as a pair of {\it Young diagrams}, which are both consisting of arrangements of boxes with left-justified rows and number of boxes per row weakly decreasing from top to bottom. Each box $b$ in such a pair of Young diagrams has a content $c(b)$, let $b$ be in the $r$-th row of its diagram and in the $c$-th column (counting from top to bottom and from left to right, starting with $1$ in both cases), then $c(b) = r-c$ if $b$ is in $\la^{(1)}$ and $c(b) = c-r$ if $b$ is in $\la^{(2)}$. Below, the pair of Young diagrams attached to the bipartition $(3,2,1,1), (2,2,1)$ is shown with the contents of each box written in the box
\newcommand{\minusone}{$-1$}
\newcommand{\minustwo}{$-2$}
\begin{eqnarray}
\label{bipartition}
\begin{minipage}{6cm}
$\left(\la^{(1)}={\young(0\minusone\minustwo,10,2,3)}, \la^{(2)}=\young(01,\minusone0,\minustwo)\right)$
\end{minipage}
\end{eqnarray}

It will be convenient to draw such a pair of Young diagrams $\varphi(\la)$ as a {\it double Young diagram} with a total of $d$ boxes in the following way.

Consider an infinite strip with $n$ columns and a horizontal line $o$. This horizontal line will split the strip into two regions, an upper and a lower part. A double Young diagram consists of two Young diagrams, one placed in the upper half with center of gravity on the lower right point of that region and a second one places in the lower part with center of gravity on the upper left  such that no column contains boxes above and below the line $o$. 

The double Young diagram attached to the weight $\la$ is constructed as follows: the Young diagram at the bottom is just the the Young diagram for $\la^{(1)}$ transposed; the Young diagram at the top  is obtained from the Young diagram for $\la^{(2)}$ by transposing the diagram and then rotating it by 180 degrees. Denote the result $((\la^{(1)})^t,{}^t(\la^{(2)}))$. The contents for the boxes are transposed respectively rotated accordingly. Below, the double Young diagram attached to the bipartition $(3,2,1,1), (2,2,1)$ is displayed:

\begin{eqnarray}
\label{displaybipartition}
&
\begin{minipage}{6cm}
$\left(\la^{(1)}={\young(0\minusone\minustwo,10,2,3)}, \la^{(2)}=\young(01,\minusone0,\minustwo)\right)$
\end{minipage}
\begin{minipage}{6cm}
\begin{tikzpicture}[scale=1, thick]
\draw[thin,red] (-.25,0)  -- +(5.5,0);

\node at (.25,1.5) {\tiny $1$};
\node at (.75,1.5) {\tiny $2$};
\node at (1.25,1.5) {\tiny $3$};
\node at (1.75,1.5) {\tiny $4$};
\node at (2.75,1.5) {\tiny $\cdots$};
\node at (3.75,1.5) {\tiny n-2};
\node at (4.25,1.5) {\tiny n-1};
\node at (4.75,1.5) {\tiny n};

\draw[dotted] (0,1.75) -- +(0,-4);
\draw[dotted] (0.5,1.75) -- +(0,-4);
\draw[dotted] (1,1.75) -- +(0,-4);
\draw[dotted] (1.5,1.75) -- +(0,-4);
\draw[dotted] (2,1.75) -- +(0,-4);
\draw[dotted] (3.5,1.75) -- +(0,-4);
\draw[dotted] (4,1.75) -- +(0,-4);
\draw[dotted] (4.5,1.75) -- +(0,-4);
\draw[dotted] (5,1.75) -- +(0,-4);

\draw rectangle (.5,-.5);
\draw (.5,0) rectangle +(.5,-.5);
\draw (0,-.5) rectangle +(.5,-.5);
\fill[fill=lightgray, draw=black] (0,-1) rectangle +(.5,-.5);
\fill[fill=lightgray, draw=black] (.5,-.5) rectangle +(.5,-.5);
\fill[fill=lightgray, draw=black] (1,0) rectangle +(.5,-.5);
\fill[fill=lightgray, draw=black] (1.5,0) rectangle +(.5,-.5);

\node at (.25,-.25) {$0$};
\node at (.75,-.25) {$1$};
\node at (1.25,-.25) {$2$};
\node at (1.75,-.25) {$3$};
\node at (.25,-.75) {$-1$};
\node at (.75,-.75) {$0$};
\node at (.25,-1.25) {$-2$};

\draw (2,-.025) rectangle +(.5,.05);
\draw (2.5,-.025) rectangle +(.5,.05);
\draw (3,-.025) rectangle +(.5,.05);

\fill[fill=lightgray, draw=black] (3.5,0) rectangle +(.5,.5);
\draw (4,0) rectangle +(.5,.5);
\fill[fill=lightgray, draw=black] (4,.5) rectangle +(.5,.5);
\draw (4.5,0) rectangle +(.5,.5);
\fill[fill=lightgray, draw=black] (4.5,.5) rectangle +(.5,.5);

\node at (3.75,.25) {$-2$};
\node at (4.25,.25) {$-1$};
\node at (4.75,.25) {$0$};
\node at (4.25,.75) {$0$};
\node at (4.75,.75) {$1$};

\\
\end{tikzpicture}
\end{minipage}
&\nonumber\\
\end{eqnarray}

In addition we will modify the {\bf content} of a box in such a double Young diagram as follows. If a box $b$ is inside the lower part of the diagram, then $c_\delta(b)=c(b) + \alpha$, while the content of a box $b$ in the upper part of the diagram is defined to be $c_\delta(b)=c(b) + \beta$.

To read off the weight $\la$ from the corresponding double Young diagram we consider the boundary boxes responsible for the shape (i.e. the boxes shaded in the example  \eqref{displaybipartition}). Let $b_i$ denote the number of boxes in column $i$, multiplied with $-1$ if the boxes are in the lower half. Then 
\begin{eqnarray*}
\lambda &=& \underline{\delta} + \sum_{i=1}^n b_i \varepsilon_i.
\end{eqnarray*}
More generally, the {\it weight}, ${\rm wt}(Y)$, of a double diagram $Y$ is defined as 
\begin{eqnarray*}
{\rm wt}(Y) &=& \sum_{i=1}^n b_i \varepsilon_i.
\end{eqnarray*}
An {\it up-down bitableaux} of length $d$ is a sequence  $\mathbf{Y}=(\mathbf{Y}_1, \mathbf{Y}_2,\ldots, \mathbf{Y}_d)$ of double Young diagrams such that $\mathbf{Y}_1$ corresponds to the trivial bipartition $(\emptyset,\emptyset)$ and two consecutive double diagrams differ just by one box (added or removed). Let ${\rm wt}(\mathbf{Y})=({\rm wt}(\mathbf{Y}))_{1\leq i\leq d}$  be the weight sequence attached to $\mathbf{Y}$. The set of all such up-down bitableaux of length $d$ is denoted by $\mathcal{T}_d$.\\

The module  $M:=\Mde \otimes V^{\otimes d}$ decomposes into a direct sum of submodules 
\begin{eqnarray}
\label{stupid}
 M&=&\bigoplus_{\mathbf{a} \in \mathbb{C}^d/S_d} M_{S_d\mathbf{a}},
\end{eqnarray}
where $a$ runs through a fixed set of representatives for the $S_d$-orbits on  $\mathbb{C}^d$ and $M_{S_d\mathbf{a}}=\oplus_{w\in S_d} M_{w(\mathbf{a})}$, with the summands defined as in \eqref{eigspaces}. These are then the subspaces of $M$ where the multiset of occurring generalized eigenvalues of all the individual $y_i$'s is fixed.

\begin{prop}
\label{surjectivity}
Assume $n \geq 2d$. There is a canonical bijection between $\mathcal{T}_d$ and the parabolic Verma modules appearing as sections in a Verma filtration of $\Mde \otimes V^{\otimes d}$ counted with multiplicities such that the following holds 
\begin{enumerate}[1.)]
\item The bijection is given by assigning to a up-down bitableau $\mathbf{Y}=(\mathbf{Y}_1, \mathbf{Y}_2,\ldots, \mathbf{Y}_d)$ the parabolic Verma module $M(\underline{\delta} + {\rm wt}(\mathbf{Y}_d))$.
\item For $1 \leq j \leq d$, let $\eta_j=1$ if $\mathbf{Y}_j$ was obtained from $\mathbf{Y}_{j-1}$ by adding a box $B_j$ and $\eta_j=-1$ if $\mathbf{Y}_j$ was obtained from $\mathbf{Y}_{j-1}$ by removing a box $B_j$. Then the parabolic Verma module $M^\mathfrak{p}(\underline{\delta} + {\rm wt}(\mathbf{Y}_d))$ associated to $\mathbf{Y}$ appears as a subquotient of the summand in \eqref{stupid} containing the generalized eigenspace of the operator $y_k$ for eigenvalue $\eta_k c_\delta(B_k)$.
\end{enumerate}
\end{prop}
\begin{proof}
In case $\delta\geq 0$ this follows directly from \cite{ES3} and the bijection between up-down bitableaux and Verma path describing the Verma modules appearing in a Verma filtration and their eigenvalues. For $\delta<0$ the arguments are totally analogous. 
\end{proof}

\begin{definition}
We call an eigenvalue of $y_k$ \emph{small} if it corresponds to the content of a box in the lower half and we call it \emph{large} if it corresponds to the content of a box in the upper half. 
\end{definition}

The idempotent we define now projects onto the generalized common eigenspaces of the elements $y_k$ that are small.

\begin{definition}
For $1 \leq k \leq d$ we denote by $\eta_k \in \VWdcycl$ the idempotent projecting onto the generalized eigenspace of $y_k$ with eigenvalue different from $\beta$. Furthermore let $\mathbf{f}_k = \eta_1 \cdots \eta_k$ and $\mathbf{f} = \mathbf{f}_d$.
\end{definition}

It is clear from  the definition resp. the calculus of up-down sequences that if $y_k$ has a large eigenvalue on some composition factor, then there exists a $j \leq k$ such that $y_j$ has eigenvalue $\beta$ on that composition factor. Hence the element $\mathbf{f}$ projects onto the common generalized eigenspaces of small eigenvalues for the commutative subalgebra generated by the $y_j$'s.\\

The idempotent $\ff$ not central, but we have at least the following formulas:

\renewcommand{\theenumi}{\alph{enumi}}

\begin{prop}\label{prop:kill_f}
In $\VWdcycl$ the following equalities hold:
\begin{enumerate}[1.)]
\item For $1 \leq k \leq d-1$ 
$$\begin{array}{rccrc}
i) & c_{k} \ff s_k \ff = c_{k} s_k \ff, & \qquad & iii) & c_k \ff e_k \ff = c_k e_k \ff,\\
ii) &b_{k+1} \ff s_k \ff = b_{k+1} s_k \ff, & \qquad & iv) &b_{k+1} \ff e_k \ff = b_{k+1} e_k \ff.
\end{array}$$
\item For $1 \leq k < d-1$
$$\begin{array}{rccrc}
i) & e_k \ff s_{k+1} \ff = e_k s_{k+1} \ff, & \qquad & iv) & \ff s_k \ff s_{k+1} \ff = \ff s_k s_{k+1} \ff,\\
ii) & e_k \ff e_{k+1} \ff = e_k e_{k+1} \ff, & \qquad & v) & \ff s_k \ff e_{k+1} \ff = \ff s_k e_{k+1} \ff,\\
iii) & \ff e_k \ff s_{k+1} \ff e_{k} \ff = \ff e_k s_{k+1} e_k \ff.
\end{array}$$
\end{enumerate}
as well as all of these equalities with $k$ and $k+1$ swapped.
\end{prop}
\begin{proof}
We start with part \emph{1.)}. For case \emph{i)} we claim that
$$ c_{k} \ff s_k \ff = c_{k} s_k \ff - c_{k} (1 - \ff) s_k \ff = c_{k} s_k \ff.$$
We only have to justify the last equality. But this holds due to Lemma \ref{lem:eigenvalue_and_s}; the image of $(1-\ff)s_k \ff$ is either 0 or consists of generalized eigenvectors for $y_{k}$ with eigenvalue $\beta$. By Corollary \ref{cor:eigenvalue_and_s} and the fact that due to the diagram calculus and the assumptions on $n$, $y_1$ always acts by a scalar these are honest eigenvectors and thus $c_k$ acts by zero.
For case \emph{ii)} we calculate
\begin{eqnarray*}
b_{k+1}\ff s_k \ff &=& b_{k+1}\ff s_k \frac{c_{k+1}}{c_{k+1}}\ff \\
&\overset{(a)}{=}& b_{k+1}\ff (c_k s_k + e_k - 1)\frac{1}{c_{k+1}}\ff \\
&\overset{(b)}{=}& b_{k+1} c_k s_k \frac{1}{c_{k+1}}\ff + b_{k+1} e_k \frac{1}{c_{k+1}}\ff- \frac{b_{k+1}}{c_{k+1}}\ff)\\
&\overset{(c)}{=}& b_{k+1} (s_k c_{k+1} - e_k + 1) \frac{1}{c_{k+1}}\ff + b_{k+1} e_k \frac{1}{c_{k+1}}\ff- \frac{b_{k+1}}{c_{k+1}}\ff = b_{k+1} s_k \ff
\end{eqnarray*}
Here equalities $(a)$ and $(c)$ hold by Lemma \ref{lem:commute_b}, while equality $(b)$ is due to the other cases of this lemma.
For case \emph{iii)} we have
$$ c_k e_k \ff = c_k \ff_{k-1} e_k \ff = c_k \ff_k \ff_{k-1} e_k \ff + c_k (1-\ff_k)\ff_{k-1} e_k \ff = c_k \ff_k e_k \ff.$$
The final equality holds because the image of $(1-\ff_k)\ff_{k-1}$ consists of eigenvectors for $y_k$ with eigenvalue $\beta$, hence are annihilated by $c_k$. That the image consists of eigenvectors follows as in case \emph{i)}. Furthermore, due to Lemma \ref{lem:eigenvalue_and_e}, $\ff_k e_k = \ff_{k+1} e_k$ and the statement follows. Case \emph{iv)} is the same since $b_{k+1}e_k = c_k e_k$ by Relation~(VW.\ref{8b}). \\

For part \emph{2.)}, we note that all of these are more or less proven in the same way using Lemma \ref{lem:eigenvalue_and_e} and Corollary \ref{cor:eigenvalue_and_s}. We will argue for $e_k \ff s_{k+1} \ff = e_k s_{k+1} \ff$ and leave the others to the reader.
It holds
$$e_k s_{k+1} \ff = e_k \ff s_{k+1} \ff + e_k(1-\ff)s_{k+1}\ff.$$
If we now look at a generalized eigenspace $M_\textbf{i}$ in the image of $(1-\ff)s_{k+1}\ff$, we see that, due to the diagram combinatorics, this can only be non-zero if $\textbf{i}_{k+1} = \beta$ and $\textbf{i}_{k+2} = -\beta$ while all other eigenvalues are small. Applying $e_k$ to this eigenspace is zero, due to Lemma \ref{lem:eigenvalue_and_e} since $\textbf{i}_{k}$ is small and thus cannot be $-\beta$. Hence only the first summand survives which proofs the claim. Similar arguments have to be applied to the other cases.
\end{proof}

Relation~(VW.\ref{7}) from Definition \ref{def:VW} on the nose or multiplied with the idempotent $\ff$ from both sides yields the following equalities:

\begin{lemma} \label{lem:commute_b}
In $\VWdcycl$ the following equalities hold for $1\leq k\leq d-1$:
\begin{enumerate}[1.)]
\item We have
$b_{k+1}s_k = s_kb_k - e_k + 1,$ and $\ff s_k \frac{1}{b_k} \ff = \frac{1}{b_{k+1}} \ff s_k \ff - \frac{1}{b_{k+1}} \ff e_k \frac{1}{b_{k}}\ff + \frac{1}{b_k b_{k+1}} \ff,$
\item We have $s_k b_{k+1}= b_ks_k - e_k + 1,$ and $\frac{1}{b_k} \ff s_k \ff = \ff s_k \frac{1}{b_{k+1}}\ff - \frac{1}{b_{k}} \ff e_k \frac{1}{b_{k+1}}\ff + \frac{1}{b_k b_{k+1}} \ff$.
\item We have $c_{k+1}s_k = s_kc_k + e_k - 1,$ and $\ff s_k \frac{1}{c_k} \ff = \frac{1}{c_{k+1}} \ff s_k \ff + \frac{1}{c_{k+1}} \ff e_k \frac{1}{c_{k}}\ff - \frac{1}{c_k c_{k+1}} \ff,$
\item We have $s_k c_{k+1}= c_ks_k + e_k - 1,$ and $\frac{1}{c_k} \ff s_k \ff = \ff s_k \frac{1}{c_{k+1}}\ff + \frac{1}{c_{k}} \ff e_k \frac{1}{c_{k+1}}\ff - \frac{1}{c_k c_{k+1}} \ff$.
\end{enumerate}
\end{lemma}

Note moreover that by Lemma \ref{lem:eigenvalue_and_e}, Corollaries \ref{cor:permute_eigenvalues} and \ref{cor:eigenvalue_and_s}, the definition of $\ff$, and our diagrammatic calculus we have that the expressions
\begin{eqnarray} 
\frac{1}{b_k} s_k \ff \text{ and } \frac{1}{b_k} e_k \ff.
\end{eqnarray}
are well-defined. This will be used in many of the proofs to come.

\section{The isomorphism theorem}
\label{sec:invers-square-roots}
In the isomorphism of the main theorem will appear some square roots. We start by recalling their definition and the required setup from \cite{SS2}. The crucial result is the following fact from \cite[Section 4]{SS2}:

\begin{lemma}\label{prop:2}
Let \(f(x) \in \mathbb{C}[x]\). Assume \(B\) is a finite-dimensional algebra, and let
  \(x_0 \in B\). Suppose that \((x-a) \nmid f(x)\) for all \(a \in \mathbb{C}\) which are generalized eigenvalues for the action of \(x_0\) on the regular representation. Then \(f(x_0)\in B\) has a (unique) inverse and a (non-unique) square root.
\end{lemma}

Let as above \(x_0 \in B\) be an element of a finite-dimensional algebra, and let \(a \in \mathbb{C}\). If \(a\) is not a generalized eigenvalue of \(x_0\) then by Proposition~\ref{prop:2} we can write expressions like
\begin{equation}
\frac{1}{x_0-a}, \qquad \sqrt{x_0-a}, \qquad  \sqrt\frac{1}{x_0 -a}.\label{eq:52}
\end{equation}
The square root is not unique, but we  make one choice once and for all, so that for example $ \sqrt{x_0-a} \sqrt\frac{1}{x_0 -a} = 1 $. For more details we refer to \cite{SS2}.

Let $1\leq k\leq d$. We fix the following elements of $\VWdcycl$
\begin{equation}
\label{19}
b_k = \beta + y_k \qquad \text{and}\qquad c_k = \beta - y_k
\end{equation}
and set (for a choice of square root)
\begin{equation}
\label{20}
Q_k = \sqrt{\frac{b_{k+1}}{b_k}} \mathbf{f}.
\end{equation}
Note that ${\frac{b_{k+1}}{b_k}} \mathbf{f}$ is well-defined (in contrast to $\frac{b_{k+1}}{b_k}$) by definition of ${b_k}$ and $ \mathbf{f}$.  By the above arguments, also the square root makes sense.

\begin{definition}
For $1\leq k \leq d-1$ define
\begin{eqnarray}
\tilde{s}_k = - Q_k s_k Q_k + \frac{1}{b_k} \mathbf{f}&\text{and}&
\tilde{e}_k = Q_k e_k Q_k.
\end{eqnarray}
\end{definition}

We finally can state our main result:

\begin{theorem}[Isomorphism theorem]\label{thm:main}
The map 
$$\Phi_\delta :\quad {\rm Br}_d(\delta) \longrightarrow \mathbf{f} \VWdcycl \mathbf{f}.$$
given on the standard generators by 
\begin{eqnarray}
\label{difficult1}
t_k \quad\longmapsto\quad -Q_k s_k Q_k + \frac{1}{b_k}\ff,&\text{and}& g_k \quad\longmapsto\quad Q_k e_k Q_k.
\end{eqnarray}
for $1 \leq k \leq d-1$ defines an isomorphism of algebras.
\end{theorem}
\begin{proof}
That the map is well-defined follows from a series of statements in Section~\ref{sec:proof_thm:main}. Namely altogether the Lemmas \ref{lem:s2=f} and \ref{lem:sisj=sjsi}, Proposition \ref{lem:si_braid}.

 Lemmas \ref{lem:e2=deltae}, \ref{lem:eiej=ejei}, \ref{lem:eee=e}, \ref{lem:es=e}, \ref{lem:se=es}, and \ref{lem:see=se} prove that the elements $\tilde{s}_k$ and $\tilde{e}_k$ for $1 \leq k < d$ satisfy all the defining relations of the Brauer algebra ${\rm Br}_d(\delta)$.
It suffices to prove surjectivity of $\Phi_\delta$, since the algebras have the same dimension, namely $(2d-1)!!$, by \cite[Proposition 4.4]{ES3}.

To prove surjectivity we use the description of a basis of  $\VWdcycl$ from \cite[Theorem 5.5]{AMR}, see \cite[Corollary 2.25]{ES3} for our special case which says in particular that any element in  $\VWdcycl$ is a linear combination of elements of the form $p_1wp_2$, where  $p_1,p_2\in\mC[y_1,\ldots,y_d]$ with degree $\leq 1$ in each variable and $w=x_1 \cdots x_r$ where  $x_j\in\{s_i,e_i\mid 1\leq i\leq d-1\}$ for $1\leq j\leq r$. We will call such a presentation $x_1 \cdots x_r$ for $w$ a {\it reduced word} if $r$ is chosen minimally to present $w$ in such a form.  

Since by Lemma \ref{lem:jucysmurphy} all the elements $y_k \ff$ are in the image $I$ of $\Phi_\delta$ it suffices to show that $\ff x_1 \cdots x_r \ff\in I$ for any reduced word $x_1 \cdots x_r$.

We show this by two inductions on the sum of the length of the word and the number of $s_i$'s occurring in the expression. For $r=0,1$ the claim is clear, since $y_k \ff\in I$ for all $k$ and so are all polynomial expressions in the $y$'s, e.g. $Q_k^{-1} \ff$ and $\frac{1}{b_k}\ff$, and thus also the elements of the form $\ff s_k \ff$ and $\ff e_k \ff$ for all $k$. Hence the claim is true for $r\leq 1$.

For $r > 1$, we first assume that the expression $\ff x_1 \cdots x_r \ff$ contains no $s_i$'s. In this case assume that $x_r = e_l$ for some $l$, then by induction we know that 
$$ \ff x_1 \cdots x_{r-1} \ff c_{l} \ff x_r \ff $$
is in the image of $\Phi_\delta$, since all three factors are in the image by induction. By Proposition \ref{prop:kill_f} we have
\begin{eqnarray}
\ff x_1 \cdots x_{r-1} \ff c_{l} \ff x_r \ff& =& \ff x_1 \cdots x_{r-1} c_{l} x_r \ff \nonumber\\
&=& \beta \ff x_1 \cdots x_{r-1} x_r \ff + \ff x_1 \cdots x_{r-1} y_{l+1} x_r \ff \nonumber\\
&=& \beta \ff x_1 \cdots x_{r-1} x_r \ff \pm \left\lbrace
\begin{array}{ll}
y_j \ff x_1 \cdots x_{r-1} x_r \ff & \text{for some } j \text{ or } \\
\ff x_1 \cdots x_{r-1} x_r \ff y_j & \text{for some } j.
\end{array}\right.
\end{eqnarray}
The last equality is possible because the word was assumed to be reduced, thus the element $y_{l+1}$ can be moved to the outside by using  repeatedly Relation~(VW.\ref{8a}) and (VW.\ref{8b}) - only creating a possible sign change. Since $\frac{1}{\beta \pm y_j} \ff\in I$, it follows that $\ff x_1 \cdots x_r \ff\in I$.
Assume now that the reduced word $\ff x_1 \cdots x_r \ff$ for some $w$ contains a positive number, say $m$, of $s$'s. Let $l$ be such that $x_r \in \{e_l,s_l\}$. Then again by induction we know that 
$$ \ff x_1 \cdots x_{r-1} \ff c_{l} \ff x_r \ff\in I. $$
 Using again Proposition \ref{prop:kill_f} we obtain a similar expression as before, namely
\begin{eqnarray*}
\ff x_1 \cdots x_{r-1} \ff c_{l} \ff x_r \ff& =& \ff x_1 \cdots x_{r-1} c_{l} x_r \ff \\
&=& \beta \ff x_1 \cdots x_{r-1} x_r \ff + \ff x_1 \cdots x_{r-1} y_{l+1} x_r \ff \\
&=& \beta \ff x_1 \cdots x_{r-1} x_r \ff \pm \left\lbrace
\begin{array}{ll}
y_j \ff x_1 \cdots x_{r-1} x_r \ff & \text{for some } j \text{ or } \\
\ff x_1 \cdots x_{r-1} x_r \ff y_j & \text{for some } j
\end{array}\right.\\
&&\hspace{3cm}+ \text{ smaller summands.}
\end{eqnarray*}
The smaller summands do not contain any $y_j$'s in this case, see Relation~(VW.\ref{7}), and are moreover either of length smaller than $r$, or of length $r$, but then with  strictly less than $m$ letters $s$'s.  Since by induction all these smaller terms are contained in $I$, the claim follows also in this case. Thus $\Phi_\delta$ is surjective and the theorem follows.
\end{proof}

\begin{remark}
\label{theremark}
  The main feature of our isomorphism is the change of the
  parameter $N$ for $\VWdcycl$ to the corresponding parameter $\delta$  of ${\rm Br}_d(\delta)$; the most important relation we have to prove is 
  \begin{eqnarray*}
  (Q_k e_k Q_k)^2= \delta \, Q_k e_k Q_k.
  \end{eqnarray*}
  Essentially, this amounts to check that
  \begin{eqnarray}
   \label{eq:30} 
  Q_k e_k \frac{b_{k+1}}{b_k} \ff e_k Q_k
  &=&  (2\beta + 1) Q_k e_k Q_k - N Q_k e_k Q_k = \delta \, Q_k e_k Q_k.
\end{eqnarray}
  see Lemma~\ref{lem:e2=deltae}. By
  \cite{Nazarov} we
  can take a formal variable \(u\) and write
  \begin{eqnarray}
    \label{eq:55}
    e_k \frac{1}{u-y_k} e_k \ff &=& \frac{W_k(u)}{u} e_k \ff,
  \end{eqnarray}
  where \(W_k(u)\) is a formal power series in
  \(u^{-1}\) as in \cite{Nazarov}. We may now be tempted to replace \(u =
  -\beta\) and be able to compute
  \(W_k(-\beta)=\beta_1\), and hence
  obtain \eqref{eq:30} from \eqref{eq:55}. Now, while this
  can be formalized in the semisimple case (by using the
  eigenvalues of \(y_k\), as done several times in
    \cite{Nazarov}), it gets much more tricky in the
  non-semisimple case. Hence we need to take another way
  using the formalism from
  Section~\ref{sec:invers-square-roots}. 
  \label{rem:5}
\end{remark}

\begin{corollary}
The algebra $\mathbf{f} \VWdcycl \mathbf{f}$ is generated by the elements $\ff s_i \ff$, $\ff e_i \ff$, and $ y_k \ff$ for $1 \leq i < d$ and $1 \leq k \leq d$.
\end{corollary}
\begin{proof}
This follows immediately from the proof of Theorem \ref{thm:main}.
\end{proof}

Before the proof the theorem we describe the preimages of the polynomial generators $y_k$. 

\subsection{The map $\Phi_\delta$ and Jucys-Murphy elements}
\label{JM}
The {\it Jucys-Murphy elements} $\xi_k$, for $1\leq k\leq d$ in the Brauer algebra ${\rm Br}_d(\delta)$ are defined as follows:
\begin{eqnarray}
 \label{jucysmurphy}
\xi_1 = 0 &\text{ and } &\xi_{k+1} = t_k \xi_{k} t_k + t_k - g_k \text{ for all } 1 < k < d. 
\end{eqnarray}

\begin{prop} \label{lem:jucysmurphy}
The map $\Phi_\delta$ from Theorem \ref{thm:main} maps $\xi_k - \alpha$ to $-y_k \ff$.
\end{prop}
\begin{proof}
We prove this by induction on $k$, for $k=1$ we have 
$$\Phi_\delta(\xi_1 - \alpha) = \Phi_\delta(-\alpha) = -\alpha \ff = -y_1 \ff.$$
For $k+1 > 1$ we calculate
\begin{eqnarray*}
\Phi_\delta(\xi_{k+1} - \alpha) &=& \Phi_\delta(t_k \xi_k t_k + t_k - g_k - \alpha)\quad =\quad \Phi_\delta(t_k (\xi_k - \alpha) t_k + t_k - g_k)\\
&=& - \tilde{s}_k y_k \tilde{s}_k + \tilde{s}_k - \tilde{e}_k\quad =\quad Q_k s_k Q_k y_k \tilde{s}_k + \tilde{s}_k - \tilde{e}_k - \frac{1}{b_k} \ff y_k \tilde{s}_k\\
&=& Q_k (y_{k+1} s_k + e_k -1 )Q_k \tilde{s}_k + \tilde{s}_k - \tilde{e}_k - \frac{1}{b_k} \ff y_k \tilde{s}_k\\
&=& y_{k+1} Q_k s_kQ_k\tilde{s}_k + \tilde{e}_k \tilde{s}_k - \frac{b_{k+1}}{b_k}\ff \tilde{s}_k + \tilde{s}_k - \tilde{e}_k - \frac{1}{b_k} \ff y_k \tilde{s}_k\\
&=& -y_{k+1} \tilde{s}_k \tilde{s}_k + \frac{y_{k+1}}{b_k}\ff \tilde{s}_k- \frac{b_{k+1}}{b_k}\ff \tilde{s}_k + \tilde{s}_k - \frac{1}{b_k} \ff y_k \tilde{s}_k\\
&=& -y_{k+1} \ff
\end{eqnarray*}
The proposition is proved.
\end{proof}

\section{Consequences: Koszulity and graded decomposition numbers}
\label{sec:consequences}
We deduce now some non-trival consequences of the main theorem. For the whole section $d$ is a positive integer and $\delta \in \mathbb{Z}$. First, one of the main results of \cite{ES3} allows us to equip the Brauer algebra with a grading. To state the result we need some additional notation for graded modules. Let $A$ be a $\mathbb{Z}$-graded algebra. For $M \in A-\operatorname{gmod}$ we denote its \emph{graded endomorphism ring} by
$$\operatorname{end}_A(M) = \bigoplus_{r \in \mathbb{Z}} \operatorname{Hom}_{A-\operatorname{gmod}}(M,M \left\langle r \right\rangle), $$
which becomes a graded ring by putting $\operatorname{end}_A(M)_r = \operatorname{Hom}_{A-\operatorname{gmod}}(M,M \left\langle r \right\rangle)$.
The composition of $f \in \operatorname{Hom}_{A-\operatorname{gmod}}(M,M \left\langle r \right\rangle)$ and $g \in \operatorname{Hom}_{A-\operatorname{gmod}}(M,M \left\langle s \right\rangle)$ is given by $g \left\langle r \right\rangle \circ f$ in the category of graded modules. Note that for a graded lift $\widehat{M} \in A-\operatorname{gmod}$ of $M \in A-\operatorname{mod}$ it holds
$$ {\rm End}_{A-\operatorname{mod}}(M) \cong \operatorname{end}_A(\widehat{M})$$
as (ungraded) algebras.

\begin{prop} \label{prop:brauergrading}
The Brauer algebra ${\rm Br}_d(\delta)$ can be equipped with a $\mathbb{Z}$-grading turning it into a $\mathbb{Z}$-graded algebra ${\rm Br}^{\rm gr}_d(\delta)$.
\end{prop}
\begin{proof}
Recall from Section \ref{sec:cycl_and_cat_O} the parabolic category $\mathcal{O}^\mathfrak{p}(n)$. Consider the endofunctor $\mathcal{F} = ? \otimes V$ of $\mathcal{O}^\mathfrak{p}(n)$. Following \cite[Sections 4]{ES3} we have the summand $\mathcal{G}$ of this functor that corresponds to projecting onto blocks with small weights (in \cite{ES3} this functor was denoted by $\widetilde{\mathcal{F}}$) such that
$$ {\rm Br}_d(\delta) \, \cong \, \ff \VWdcycl \ff \, \cong \, \ff {\rm End}_{\mathfrak{g}}( \mathcal{F}^d M^\mathfrak{p}(\underline{\delta})) \ff \, \cong \, {\rm End}_{\mathfrak{g}}( \mathcal{G}^d M^\mathfrak{p}(\underline{\delta})),$$
as algebras.
By choosing a minimal projective generator $P$ of $\mathcal{O}^\mathfrak{p}(n)$ we have an equivalence of categories 
$$\mathcal{O}^\mathfrak{p}(n) \cong {\rm mod}-{\rm End}_\mathfrak{g}(P).$$
Following \cite{ES2} we equip $A:={\rm End}_\mathfrak{g}(P)$ with a Koszul grading and denote by $\widehat{\mathcal{O}}^\mathfrak{p}(n):=A-\operatorname{gmod}$ its associated category of graded modules. In \cite[Section 5]{ES3} a graded lift $\widehat{\mathcal{F}}$ of $\mathcal{F}$ is constructed by choosing graded lifts for each summand obtained by projecting onto blocks, see \cite[Lemma 5.3]{ES3}. Thus it also yields a graded lift $\widehat{\mathcal{G}}$ of $\mathcal{G}$ and gives
$$ {\rm Br}_d^{\rm gr}(\delta) \, := \, {\rm end}_A( \widehat{\mathcal{G}}^d \widehat{M^\mathfrak{p}(\underline{\delta})}),$$
where $\widehat{M^\mathfrak{p}(\underline{\delta})}$ is the standard graded lift of the parabolic Verma module.
\end{proof}

With this grading one can establish Koszulity.

\begin{theorem} \label{thm:koszul}
The Brauer algebra ${\rm Br}^{\rm gr}_d(\delta)$ is Morita equivalent to a Koszul algebra if and only if $\delta \neq 0$ or $\delta=0$ and $d$ odd.
\end{theorem}
\begin{proof}
This follows directly from our main Theorem \ref{thm:main} together with \cite[Theorem 5.1]{ES3}.
\end{proof}

\begin{prop} \label{thm:cellularity}
The Brauer algebra ${\rm Br}^{\rm gr}_d(\delta)$ is graded cellular.
\end{prop}
\begin{proof}
It follows from \cite{Koenig-Xi} that the idempotent truncation of the quasi-hereditary algebra $\VWdcycl$ is always cellular.
\end{proof}

\begin{theorem} \label{thm:quasihereditary}
The Brauer algebra ${\rm Br}^{\rm gr}_d(\delta)$ is graded quasi-hereditary if and only if $\delta\not=0$ or $\delta=0$ and $d$ odd.
\end{theorem}

\begin{proof}
By \cite[Theorem 4.13 and Remark 4.14]{ES3} the highest weight structure of $\VWdcycl$ induces a highest weight structure on $\ff \VWdcycl \ff$ if and only if $\delta\not=0$ or $\delta=0$ and $d$ odd. Since by \cite[Definition 4.11]{ES3} the labelling posets of standard modules for $\ff \VWdcycl \ff$ agrees with the one for the Brauer algebra from \cite{CDM} the isomorphism from Theorem \ref{thm:main} is an isomorphism of quasi-hereditary algebras.

The result then follows from the general theory of graded category $\mathcal{O}$ (see \cite{Stroppel}) using \cite[Theorem 4.9]{ES3}.
\end{proof}

Denote by $\Delta(\lambda)$ for $\lambda \in \Lambda_d$ the standard module for ${\rm Br}_d(\delta)$ and by $L(\lambda)$ for $\lambda \in \Lambda^\delta_d$ the corresponding simple quotient, see \cite{CDM}. As in the introduction denote by $F$ the grading forgetting functor from ${\rm Br}^{\rm gr}_d(\delta)-\operatorname{gmod}$ to ${\rm Br}_d(\delta)-\operatorname{mod}$.

\begin{theorem}\label{thm:liftmodules}
Assume that $\delta \neq 0$ or $\delta = 0$ and $d$ is odd, i.e. the case where ${\rm Br}^{\rm gr}_d(\delta)$ is graded quasi-hereditary. For any $\la\in\La_d$ there exists a unique modules $\widehat{\Delta}(\la) \in {\rm Br}^{\rm gr}_d(\delta)-\operatorname{gmod}$ such that $F\widehat{\Delta}(\la)\cong {\Delta}(\la)$ and for $\lambda \in \Lambda_d^\delta$ there exists a unique $\widehat{L}(\la) \in {\rm Br}^{\rm gr}_d(\delta)-\operatorname{gmod}$ such that $F\widehat{L}(\la)\cong L(\la)$,
and both modules are concentrated in non-negative degrees with non-vanishing degree zero.
\end{theorem}
\begin{proof}
From \cite[Lemma 1.5]{Stroppel} it follows that graded lifts, if they exist, are unique up to isomorphism and grading shifts. With the assumption on the degree of the modules the graded lifts will be unique up to isomorphism in our case.

For $\VWdcycl$ the existence of graded lifts follows from \cite[Theorem 4.9]{ES3} and general theory of category $\mathcal{O}$, see \cite{Stroppel}.

The existence of the graded lifts for $\ff \VWdcycl \ff$ then follows from \cite[Theorem 4.13]{ES3} via a quotient functor construction.
\end{proof}

We now want to match the multiplicities of simple modules occurring in a standard module with the coefficients of certain Kazhdan-Lusztig polynomials. Denote by $W$ the Weyl group of $\mathfrak{g}$ and by $W_\mathfrak{p}$ the parabolic subgroup generated by all simple roots except $\alpha_0$, i.e., the one corresponding to the parabolic $\mathfrak{p}$ from the introduction and Section \ref{sec:cycl_and_cat_O}. By $W^\mathfrak{p}$ we denote the shortest coset representatives in $W_\mathfrak{p} \backslash W$. For $x,y \in W^\mathfrak{p}$ let $n_{x,y}(q) \in \mathbb{Z}[q]$ be the parabolic Kazhdan-Lusztig polynomial of type $({\rm D}_n,{\rm A}_{n-1})$, see \cite{Boe} and for this special case \cite{LS}.

Given now $\nu \in X_n^\mathfrak{p}$ there is a unique $\nu_{\rm dom} \in X_n^\mathfrak{p}$ and $x_\nu \in W^\mathfrak{p}$ such that $\nu_{\rm dom} + \rho$ is dominant and
$$ x_\nu(\nu_{\rm dom} + \rho) = \nu + \rho.$$

We now give a dictionary how to translate between the labelling set of standard modules and Weyl group elements. To a partition $\lambda$ we associate a double Young diagram $Y(\lambda)$ via the bipartition $(\lambda,\emptyset)$ and a weight ${\rm wt}(\lambda)=\underline{\delta} + {\rm wt}(Y(\lambda))$. For $\lambda,\mu \in \La_d$ we put
$$ 
n_{\lambda,\mu}(q) = \left\lbrace \begin{array}{ll}
n_{x_{{\rm wt}(\lambda)}, x_{{\rm wt}(\mu)}}(q) & \text{if } {\rm wt}(\lambda) + \rho \in W \cdot ({\rm wt}(\mu)+\rho), \\
0 & \text{otherwise.}
\end{array}\right.
$$

\begin{theorem}\label{thm:multiplicities}
For $\lambda \in \La_d$, the module $\widehat{\Delta}(\la)$ has a Jordan-H\"older series in ${\rm Br}^{\rm gr}_d(\delta)-\operatorname{gmod}$ with multiplicities  given by 
\begin{eqnarray*}
\left[\widehat{\Delta}(\la)\, :\, \widehat{L}(\mu)<i>\right]&=&n_{\lambda,\mu,i},
\end{eqnarray*}
where $n_{\lambda,\mu}(q)=\sum_{i\geq 0}n_{\lambda,\mu,i} q^i$ and $\mu \in \La_d^\delta$.
\end{theorem}
\begin{proof}
Denote by $\widetilde{\Delta}(\la)$ the standard module in $\VWdcycl-\operatorname{gmod}$ that is sent to $\widehat{\Delta}(\la)$ via the quotient functor $\mathcal{Q}$ to 
$\ff \VWdcycl \ff -\operatorname{gmod}$ and analogously $\widetilde{L}(\mu)$ the simple module in $\VWdcycl-\operatorname{gmod}$. Then the following equality holds
$$
n_{\lambda,\mu,i} \overset{(a)}{=} \left[\widetilde{\Delta}(\la)\, :\, \widetilde{L}(\mu)<i>\right] \overset{(b)}{=} \left[\mathcal{Q}\widetilde{\Delta}(\la)\, :\, \mathcal{Q}\widetilde{L}(\mu)<i>\right] \overset{(c)}{=} \left[\widehat{\Delta}(\la)\, :\, \widehat{L}(\mu)<i>\right],
$$
where $(a)$ is due to \cite[Theorem 4.9]{ES3}, $(b)$ is due to \cite[p.136 (4)(vii)]{Donkin}, and $(c)$ is due to \cite[p.136 (4)(iv)]{Donkin} for the simple module and \cite[Prop. 4.3]{Koenig-Xi} for the standard module.
\end{proof}

\section{Well-definedness of $\Phi_\delta$} \label{sec:proof_thm:main}

In this section we establish the remaining part of Theorem \ref{thm:main}, namely the map $\Phi_\delta$ being well-defined. In other words, we have to verify that the $\tilde{s}_i$ and $\tilde{e}_i$ satisfy the Brauer relations. We will commence with a few lemmas that will help in the calculations, allowing us to simplify various expressions.

\begin{lemma}
It holds $e_k \eta_{k+1} = e_k \eta_k$ for all $k$, hence $e_k \mathbf{f}_{k+1} = e_k \mathbf{f}_k$.
\end{lemma}
\begin{proof}
Let $\VWdcycl$ act on itself by the regular representation. By Lemma \ref{lem:eigenvalue_and_e} it follows that if either side of the equation acts non-trivially then the eigenvalues of $y_k$ and $y_{k+1}$ are opposite. Hence if one is small so is the other.
\end{proof}

We often need to simplify expressions involving fractions, such as in the definitions of $\tilde{s}_k$ and $\tilde{e}_k$. The following proposition collects a few useful formulas for this.  

\begin{prop} \label{prop:reduce_fracs}
In $\VWdcycl$ the following equalities hold for $1 \leq k \leq d-1$:
$$\begin{array}{rccrc}
i) & e_k \frac{1}{b_k} s_k \ff= \frac{1}{2\beta}e_k \frac{1}{b_k}\ff, & \qquad & iii) & e_k \frac{1}{b_k} e_k \ff = (1 + \frac{1}{2\beta})e_k \ff,\\
ii) & \ff s_k \frac{1}{b_k} e_k \ff= \frac{1}{2\beta} \frac{1}{b_k} \ff e_k \ff.
\end{array}$$
\end{prop}
\begin{proof}
Note first that $\frac{1}{b_k}s_k\ff$ and $\frac{1}{b_k}e_k\ff$ are well-defined. To see this we look at the action of these elements on the module $M=\MdV$. For both, $s_k\ff$ and $e_k\ff$ it holds that their image is contained in those subspaces $M_{\textbf{i}}$ such that $i_k \neq - \beta$, due to Lemma \ref{lem:eigenvalue_and_e} and Lemma \ref{lem:eigenvalue_and_s} in conjunction with the diagram calculus, hence $\frac{1}{b_k}$ is defined as an endomorphism of these images.

Let us first assume $iii)$ is proven already. To verify $i)$ we calculate
\begin{eqnarray*}
e_k \frac{1}{b_k}s_k \ff &\overset{(a)}{=}& e_k s_k \frac{1}{b_{k+1}}\ff - e_k \frac{1}{b_k}e_k \frac{1}{b_{k+1}}\ff + e_k \frac{1}{b_k b_{k+1}}\ff\\
&\overset{(b)}{=}& e_k s_k \frac{1}{b_{k+1}}\ff - (1 + \frac{1}{2\beta}) e_k \frac{1}{b_{k+1}}\ff + e_k \frac{1}{b_k b_{k+1}}\ff\\
&=& e_k \left( \frac{2\beta - b_k}{2\beta b_k b_{k+2}} \right) = e_k \left( \frac{c_k}{2\beta b_k b_{k+1}} \right) 
= e_k \left( \frac{b_{k+1}}{2\beta b_k b_{k+1}} \right) \\
&=& \frac{1}{2\beta} e_k \frac{1}{b_k}\ff
\end{eqnarray*}
where equality $(a)$ holds by Lemma \ref{lem:commute_b} and equality $(b)$ is valid thanks to part $iii)$ of this proposition.

Formula $ii)$ is shown analogously, but note that since $\frac{1}{b_{k+1}}e_k$ is in general not defined we have to multiply the whole equation by $\ff$ from the left to make it well-defined.

Finally let us consider the formula $iii)$ which we will prove by induction on $k$. If $k=1$ then  note that by Definition \ref{defalphabeta} the element $y_1$ has exactly two eigenvalues, namely $\alpha$ and $\beta$ as in \eqref{ab}, with the projections $\frac{y_1-\beta}{\alpha-\beta}$ respectively $\frac{y_1-\alpha}{\beta-\alpha}$ onto the eigenspaces. Then we obtain 
\begin{eqnarray*}
&&e_1 \frac{1}{b_1} e_1\ff\\
& =& e_1\left (\frac{1}{\alpha+\beta}\frac{y_1-\beta}{\alpha-\beta}+ \frac{1}{2\beta}\frac{y_1-\alpha}{\beta-\alpha}\right)e_1\ff
\quad=\quad \frac{1}{\alpha-\beta}e_1\left(\frac{y_1-\beta}{\alpha+\beta}+\frac{\alpha-y_1}{2\beta}\right)e_1 \ff\\
&=&
e_1\frac{1}{(\alpha+\beta)2\beta}\left(-y_1+\alpha+2\beta\right)e_1 \ff
\quad\stackrel{\eqref{ab}}{=}
\quad\frac{1}{N\beta}\left(-e_1y_1e_1+(\alpha+2\beta) e_1^2\right)\ff\\
&=&\frac{1}{2N\beta}\left(N(1-N)+N^2+2\beta\right)\ff
=\left(1 + \frac{1}{2\beta}\right)e_1 \ff. 
\end{eqnarray*}
Now assume the formula holds for $k$ and, by applying Lemma \ref{lem:commute_b} repeatedly, we obtain
\begin{eqnarray} \label{puh}
&& e_{k+1} \frac{1}{b_{k+1}} e_{k+1}\ff\quad=\quad e_{k+1}s_{k}(s_{k} \frac{1}{b_{k+1}}) e_{k+1}\ff\nonumber\\
 &{=}& 
    e_{k+1}s_{k}\left(\frac{1}{b_k}\right)s_{k}e_{k+1}\ff\
 + e_{k+1}s_k\left(\frac{1}{b_k}e_k\frac{1}{b_{k+1}}\right)e_{k+1}\ff
 -
 e_{k+1}s_k\left(\frac{1}{b_k}\frac{1}{b_{k+1}}\right)e_{k+1}\ff
\end{eqnarray}
by Lemma \ref{lem:commute_b}. Now the first summand in \eqref{puh} equals, by (VW.\ref{6c}),
\begin{eqnarray*}
e_{k+1}e_{k}s_{k+1}\left(\frac{1}{b_k}\right)s_{k+1}e_{k}e_{k+1}\ff
&=&e_{k+1}e_{k}\frac{1}{b_k}e_{k}e_{k+1}\ff
=\left(1+\frac{1}{2\beta}\right)e_{k+1}e_ke_{k+1}\ff\\
&=&\left(1+\frac{1}{2\beta}\right)e_{k+1}\ff
\end{eqnarray*}
by induction, whereas the second summand equals 
\begin{eqnarray*}
&&e_{k+1}s_k\frac{1}{b_k}e_k\frac{1}{b_{k+1}}e_{k+1}\ff\\
&=&e_{k+1}\left(s_k\frac{1}{b_k}\right)e_ke_{k+1}\frac{1}{c_k}\ff\\
&=&e_{k+1}\left(\frac{1}{b_{k+1}}s_k\right)e_ke_{k+1}\frac{1}{c_k}\ff
-e_{k+1}\left(\frac{1}{b_{k+1}}e_k\frac{1}{b_k}\right) e_ke_{k+1}\frac{1}{c_k}\ff
+e_{k+1}\left(\frac{1}{b_{k+1}}\frac{1}{b_k}\right) e_ke_{k+1}\frac{1}{c_k}\ff\\
&=&
\frac{1}{c_k}e_{k+1}e_ke_{k+1}\frac{1}{c_{k}}\ff
-\left(1+\frac{1}{2\beta}\right)\frac{1}{c_{k}} e_{k+1}e_ke_{k+1}\frac{1}{c_{k}}\ff
+\frac{1}{b_{k}c_k} e_{k+1}e_ke_{k+1}\frac{1}{c_k}\ff\\
&=&-\frac{1}{2\beta} \frac{1}{c_k^2}e_{k+1}\ff + \frac{1}{b_kc_k^2}e_{k+1}\ff\\
\end{eqnarray*}
by (VW.\ref{8a}) and (VW.\ref{5a}), Lemma \ref{lem:commute_b} and induction hypothesis. 
Finally the third summand in \eqref{puh} equals 
\begin{eqnarray*}
&&
-e_{k+1}\left(s_k\frac{1}{b_{k+1}}\right)e_{k+1}\frac{1}{b_k}\ff \\
&=&
-e_{k+1}\left(\frac{1}{b_{k}s_k}\right)e_{k+1}\frac{1}{b_k}\ff
-e_{k+1}\left(\frac{1}{b_{k}}e_k\frac{1}{b_{k+1}}\right)e_{k+1}\frac{1}{b_k}\ff
+e_{k+1}\left(\frac{1}{b_kb_{k+1}}\right)e_{k+1}\frac{1}{b_{k}}\ff\\
&=&-\frac{1}{b_{k}^2}e_{k+1}\ff-\frac{1}{b_{k}^2c_k} e_{k+1}\ff+
\frac{1}{b_{k}^2} e_{k+1} \frac{1}{b_{k+1}} e_{k+1}\ff.
\end{eqnarray*}
Hence altogether we obtain 
\begin{eqnarray*}
&&\left(1-\frac{1}{b_{k}^2}\right) e_{k+1} \frac{1}{b_{k+1}} e_{k+1} \ff =\left(1+\frac{1}{2\beta}\right)e_{k+1}\ff + \left( -\frac{1}{2\beta} \frac{1}{c_k^2} + \frac{1}{b_kc_k^2}-\frac{1}{b_{k}^2}-\frac{1}{b_{k}^2c_k}\right) e_{k+1}\ff\\
&=&\left(1-\frac{1}{b_{k}^2}\right)\left(1+\frac{1}{2\beta}\right)e_{k+1}\ff
 + \left(\frac{1}{2\beta}\frac{1}{b_{k}^2} -\frac{1}{2\beta} \frac{1}{c_k^2} + \frac{1}{b_kc_k^2}-\frac{1}{b_{k}^2c_k} \right) e_{k+1}\ff.
\end{eqnarray*}
Now one easily checks that the last coefficient in front of the final $e_{k+1}\ff$ is zero and since  $\left(1-\frac{1}{b_{k}^2}\right)$ is invertible on the image of $e_{k+1}\ff$ we obtain
\begin{eqnarray*}
  e_{k+1} \frac{1}{b_{k+1}} e_{k+1} \ff&=&\left(1+\frac{1}{2\beta}\right)e_{k+1} \ff
\end{eqnarray*}
which finishes the proof.
\end{proof}

\subsection{The key relation $\ff \VWdcycl \ff$}
The following is the most crucial point of the proof (see also Remark~\ref{theremark}):

\begin{lemma} \label{lem:e2=deltae}
We have $\tilde{e}_k^2 = \delta \tilde{e}_{k}$ for $1 \leq k \leq d-1$.
\end{lemma}
\begin{proof}
We compute
\begin{equation} \label{eqn:e2=e}
\begin{aligned}
\tilde{e}_k^2 &= Q_k e_k \frac{b_{k+1}}{b_k} \ff e_k Q_k \overset{(a)}{=} Q_k e_k \frac{c_k}{b_k} \ff e_k Q_k \overset{(b)}{=} Q_k e_k \frac{c_k}{b_k} e_k Q_k \\
&\overset{(c)}{=} 2 \beta Q_k e_k \frac{1}{b_k}e_k Q_k - Q_k e_k^2 Q_k \overset{(d)}{=} (2\beta + 1) Q_k e_k Q_k - N Q_k e_k Q_k = \delta \tilde{e_k}.
\end{aligned}
\end{equation}
Where equality $(a)$ follows from Relation~(VW.\ref{8b}), equality $(b)$ holds by Proposition \ref{prop:kill_f}, equality $(c)$ just expands $c_k$ as $2\beta - b_k$ and equality $(d)$ is valid thanks to Proposition~\ref{prop:reduce_fracs}.
\end{proof}

\subsection{Symmetric group relations in $\ff \VWdcycl \ff$}
In this part we show that the $\tilde{s}_j$'s satisfy the defining relation of the symmetric group.

\begin{lemma} \label{lem:s2=f}
We have $\tilde{s}_k^2 = \mathbf{f}$ for $1 \leq k \leq d-1$.
\end{lemma}
\begin{proof}
We compute
\begin{align*}
\tilde{s}_k^2 =\,& Q_k \left( s_k \frac{b_{k+1}}{b_k}\ff s_k - s_k \frac{1}{b_k} - \frac{1}{b_k} \ff s_k \right) Q_k + \frac{1}{b_k^2}\ff \\
\overset{(a)}{=}&Q_k \left( s_k \frac{b_{k+1}}{b_k} s_k - s_k \frac{1}{b_k} - \frac{1}{b_k} s_k \right) Q_k + \frac{1}{b_k^2}\ff \\
\overset{(b)}{=}& Q_k \left( s_k b_{k+1} s_k \frac{1}{b_{k+1}} - s_k \frac{b_{k+1}}{b_k} e_k \frac{1}{b_{k+1}} - \frac{1}{b_k} s_k \right) Q_k + \frac{1}{b_k^2}\ff\\
\overset{(c)}{=}& Q_k \left( \frac{b_k}{b_{k+1}} + (s_k -  e_k) \frac{1}{b_{k+1}} - s_k \frac{b_{k+1}}{b_k} e_k \frac{1}{b_{k+1}} - \frac{1}{b_k} s_k \right) Q_k + \frac{1}{b_k^2}\ff\\
\overset{(d)}{=}&  Q_k \left( \frac{b_k}{b_{k+1}} + (s_k -  e_k) \frac{1}{b_{k+1}} - b_k s_k \frac{1}{b_k} e_k \frac{1}{b_{k+1}} + e_k \frac{1}{b_k} e_k \frac{1}{b_{k+1}} - \frac{1}{b_k} e_k \frac{1}{b_{k+1}} - \frac{1}{b_k} s_k \right) Q_k\\
& + \frac{1}{b_k^2}\ff\\ 
\overset{(e)}{=}&  Q_k \left( \frac{b_k}{b_{k+1}} + (s_k -  e_k) \frac{1}{b_{k+1}} - \frac{1}{2\beta} e_k \frac{1}{b_{k+1}} + (1+\frac{1}{2\beta}) e_k \frac{1}{b_{k+1}} - \frac{1}{b_k} e_k \frac{1}{b_{k+1}} - \frac{1}{b_k} s_k \right) Q_k\\
& + \frac{1}{b_k^2}\ff\\
\overset{(f)}{=}&  Q_k \left( \frac{b_k}{b_{k+1}} -  \frac{1}{b_kb_{k+1}} \right) \ff Q_k + \frac{1}{b_k^2}\ff = \ff\\
\end{align*}
where $(a)$ follows from Proposition \ref{prop:kill_f} and the fact that $\ff$ commutes with $\frac{1}{b_k}$ on the image of $s_k$, $(b)$ is due to Lemma \ref{lem:commute_b}, $(c)$ and $(d)$ are applications of Relation~(VW.\ref{7}), $(e)$ uses Proposition \ref{prop:reduce_fracs}, and finally $(f)$ uses again Lemma \ref{lem:commute_b}.
\end{proof}

Since $\ff s_i \ff$ commutes with $Q_j$ and $\frac{1}{b_{j}} \ff$ when $| i - j | > 1$ the following lemma holds.

\begin{lemma} \label{lem:sisj=sjsi}
We have $\tilde{s}_i \tilde{s}_j= \tilde{s}_j \tilde{s}_i$ for $1 \leq i,j \leq d-1$ with $|i-j| > 1$.
\end{lemma}

We verify now the braid relations, which is a surprisingly non-trivial task.

\begin{prop} \label{lem:si_braid}
The braid relation $\tilde{s}_i \tilde{s}_{i+1} \tilde{s}_i= \tilde{s}_{i+1} \tilde{s}_i \tilde{s}_{i+1}$ holds for $1 \leq i < d-1$.
\end{prop}
\begin{proof}
We expand both sides of the equality below and compute\\
\noindent\begin{minipage}{.57\linewidth}
\begin{align}
& \tilde{s}_k \tilde{s}_{k+1} \tilde{s}_k = \nonumber \\
& -Q_k s_k Q_k Q_{k+1} s_{k+1} Q_{k+1} Q_k s_k Q_k \label{eqn:lhs1} \\
& + \frac{1}{b_k} \ff Q_{k+1} s_{k+1} Q_{k+1} Q_k s_k Q_k \label{eqn:lhs2}\\
& + Q_k s_k Q_k \frac{1}{b_{k+1}} \ff Q_k s_k Q_k \label{eqn:lhs3}\\
& + Q_k s_k Q_k Q_{k+1} s_{k+1} Q_{k+1} \frac{1}{b_k} \ff \label{eqn:lhs5}
\end{align}
\end{minipage}
\begin{minipage}{.40\linewidth}
\begin{align}
& \notag \\
& - \frac{1}{b_k b_{k+1}} \ff Q_k s_k Q_k  \label{eqn:lhs4} \\
& - \frac{1}{b_k} \ff Q_{k+1} s_{k+1} Q_{k+1} \frac{1}{b_k} \ff \label{eqn:lhs6}\\
& - Q_k s_k Q_k \frac{1}{b_k b_{k+1}} \ff \label{eqn:lhs7}\\
& + \frac{1}{b_k^2 b_{k+1}} \ff \label{eqn:lhs8}.
\end{align}
\end{minipage}\\ \noindent
and \\ \noindent
\begin{minipage}{.57\linewidth}
\begin{align}
&\tilde{s}_{k+1} \tilde{s}_{k} \tilde{s}_{k+1} = \nonumber\\
&  -Q_{k+1} s_{k+1} Q_{k+1} Q_{k} s_{k} Q_{k} Q_{k+1} s_{k+1} Q_{k+1} \label{eqn:rhs1} \\
& + \frac{1}{b_{k+1}} \ff Q_{k} s_{k} Q_{k} Q_{k+1} s_{k+1} Q_{k+1} \label{eqn:rhs2}\\
& + Q_{k+1} s_{k+1} Q_{k+1} \frac{1}{b_{k}} \ff Q_{k+1} s_{k+1} Q_{k+1} \label{eqn:rhs3}\\
& + Q_{k+1} s_{k+1} Q_{k+1} Q_{k} s_{k} Q_{k} \frac{1}{b_{k+1}} \ff \label{eqn:rhs5}
\end{align}
\end{minipage}
\begin{minipage}{.40\linewidth}
\begin{align}
& \nonumber \\
& - \frac{1}{b_k b_{k+1}} \ff Q_{k+1} s_{k+1} Q_{k+1}  \label{eqn:rhs4}\\
& - \frac{1}{b_{k+1}} \ff Q_{k} s_{k} Q_{k} \frac{1}{b_{k+1}} \ff \label{eqn:rhs6}\\
& - Q_{k+1} s_{k+1} Q_{k+1} \frac{1}{b_{k} b_{k+1}} \ff \label{eqn:rhs7}\\
&  + \frac{1}{b_k b_{k+1}^2} \ff \label{eqn:rhs8}.
\end{align}
\end{minipage}\\ \noindent
For the following calculations we abbreviate $A:=Q_{k+1} \frac{1}{\sqrt{b_k}} \ff$. To improve readability we highlight those terms that are modified in each step, using either Lemma \ref{lem:commute_b} to move terms past $s_j$'s, Proposition \ref{prop:kill_f} to eliminate $\ff$'s or Proposition \ref{prop:reduce_fracs} to modify terms involving fractions.

We first simplify some parts of the right hand side of the braid.
\begin{align}\label{eqn:14g+14h}
& \text{(\ref{eqn:rhs7})}+\text{(\ref{eqn:rhs8})} = - A \highlightbox{s_{k+1} \frac{1}{b_{k+1}}} \ff A + \frac{1}{b_kb_{k+1}^2} \ff \notag \\
=& - A \left( \frac{1}{b_{k+2}}\ff s_{k+1} - \frac{1}{b_{k+2}} \ff e_{k+1} \frac{1}{b_{k+1}}+ \frac{1}{b_{k+1}b_{k+2}}\right) A + \frac{1}{b_k b_{k+1}^2} \ff\\
=& A \frac{1}{b_{k+2}} \ff e_{k+1} \frac{1}{b_{k+1}}A - A \frac{1}{b_{k+2}}\ff s_{k+1} A\notag 
\end{align}

\begin{align}\label{eqn:14c+14d}
& \text{(\ref{eqn:rhs3})}+\text{(\ref{eqn:rhs4})} = A \left( s_{k+1} \frac{b_{k+2}}{b_{k+1}} \highlightbox{\ff} s_{k+1} - \frac{1}{b_{k+1}}\highlightbox{\ff} s_{k+1} \right) A \notag \\
=& A \left( \highlightbox{s_{k+1}b_{k+2}} \frac{1}{b_{k+1}} s_{k+1} - \frac{1}{b_{k+1}}s_{k+1} \right) A \notag \\
=& A \left( \left(b_{k+1}s_{k+1} - e_{k+1} + 1 \right) \frac{1}{b_{k+1}} s_{k+1} - \frac{1}{b_{k+1}}s_{k+1} \right) A \notag \\
=& A \left( b_{k+1}s_{k+1} \highlightbox{\frac{1}{b_{k+1}} s_{k+1}} - \highlightbox{e_{k+1}\frac{1}{b_{k+1}} s_{k+1}} \right) A\notag \\
=& A \left( b_{k+1}s_{k+1} \left( s_{k+1} \frac{1}{b_{k+2}} - \frac{1}{b_{k+1}} e_{k+1} \frac{1}{b_{k+2}} + \frac{1}{b_{k+1}b_{k+2}} \right)- \frac{1}{2\beta} e_{k+1}\frac{1}{b_{k+1}} \right) A\notag \\
=& A \left( \frac{b_{k+1}}{b_{k+2}} - b_{k+1}\highlightbox{s_{k+1}\frac{1}{b_{k+1}} e_{k+1}} \frac{1}{b_{k+2}} + b_{k+1}s_{k+1}\frac{1}{b_{k+1}b_{k+2}} - \frac{1}{2\beta} e_{k+1}\frac{1}{b_{k+1}} \right) A\\
=& A \left( \frac{b_{k+1}}{b_{k+2}} - \frac{1}{2\beta} e_{k+1} \frac{1}{b_{k+2}} + b_{k+1}s_{k+1}\frac{1}{b_{k+1}b_{k+2}} - \frac{1}{2\beta} e_{k+1}\frac{1}{b_{k+1}} \right) A\notag \\
=& A \left( \frac{b_{k+1}}{b_{k+2}} - e_{k+1} \frac{1}{b_{k+1}b_{k+2}} + \highlightbox{b_{k+1}s_{k+1}}\frac{1}{b_{k+1}b_{k+2}} \right) A\notag \\
=& A \left( \frac{b_{k+1}}{b_{k+2}} - e_{k+1} \frac{1}{b_{k+1}b_{k+2}} + \left(s_{k+1} b_{k+2} + e_{k+1} - 1 \right) \frac{1}{b_{k+1}b_{k+2}} \right) A\notag \\
=& A \left( \frac{b_{k+1}}{b_{k+2}} + s_{k+1} \frac{1}{b_{k+1}} - \frac{1}{b_{k+1}b_{k+2}} \right) A = \frac{1}{b_{k}} \ff  + A \highlightbox{s_{k+1} \frac{1}{b_{k+1}}} A - \frac{1}{b_{k}b_{k+1}^2} \ff\notag \\
=& \frac{1}{b_{k}} \ff  + A \left(\frac{1}{b_{k+2}} \ff s_{k+1} \right)A - A \left(\frac{1}{b_{k+2}}\ff e_{k+1} \frac{1}{b_{k+1}} \right)A\notag 
\end{align}

Adding \eqref{eqn:14g+14h} and \eqref{eqn:14c+14d} we obtain:
\begin{equation} \label{eqn:14c+14d+14g+14h}
\text{(\ref{eqn:rhs3})}+\text{(\ref{eqn:rhs4})} + \text{(\ref{eqn:rhs7})}+\text{(\ref{eqn:rhs8})}= \frac{1}{b_{k}} \ff
\end{equation}

On the left hand side we simplify the following.
\begin{align}\label{eqn:-13c-13d}
& -\text{(\ref{eqn:lhs3})}-\text{(\ref{eqn:lhs4})} = A \left( -\frac{b_{k+1}}{b_{k+2}} \highlightbox{\ff} s_k \frac{1}{b_k} \ff s_k b_{k+1} + \frac{1}{b_kb_{k+2}} \ff s_k b_{k+1} \right) A \notag \\
=& A \left( - \highlightbox{b_{k+1} s_k} \frac{1}{b_k} \ff s_k \frac{b_{k+1}}{b_{k+2}} + \frac{1}{b_k} \ff s_k \frac{b_{k+1}}{b_{k+2}} \right) A \notag \\
=& A \left( - \left( s_k b_k - e_k + 1 \right) \frac{1}{b_k} \ff s_k \frac{b_{k+1}}{b_{k+2}} + \frac{1}{b_k} \ff s_k \frac{b_{k+1}}{b_{k+2}} \right) A \\
=& A \left( - s_k \ff s_k \frac{b_{k+1}}{b_{k+2}} + e_k \frac{1}{b_k}\ff s_k \frac{b_{k+1}}{b_{k+2}} \right) A \notag \\
=& A \left( - \frac{b_{k+1}}{b_{k+2}} + s_k (1-\ff) s_k \frac{b_{k+1}}{b_{k+2}} + e_k \frac{1}{b_k} s_k \frac{b_{k+1}}{b_{k+2}} - e_k \frac{1}{b_k} (1-\ff) s_k \frac{b_{k+1}}{b_{k+2}}\right) A \notag \\
=& - \frac{1}{b_{k}} + A s_k (1-\ff) s_k \frac{b_{k+1}}{b_{k+2}}A + A \frac{1}{2\beta} e_k \frac{b_{k+1}}{b_k b_{k+2}} A - A e_k \frac{1}{b_k} (1-\ff) s_k \frac{b_{k+1}}{b_{k+2}}A \notag
\end{align}

Furthermore, we simplify the following.

\begin{equation} \label{eqn:14f-13g-13h}
\begin{aligned}
& \text{(\ref{eqn:rhs6})}-\text{(\ref{eqn:lhs7})}-\text{(\ref{eqn:lhs8})} = - \frac{1}{b_{k+1}} \ff Q_{k} s_{k} \frac{Q_{k}}{b_{k+1}}\ff + Q_k \highlightbox{s_k \frac{1}{b_k}} \frac{Q_{k}}{b_{k+1}} \ff - \frac{1}{b_k^2b_{k+1}}\ff \\
=& - \frac{1}{b_{k+1}} \ff Q_{k} s_{k} \frac{Q_{k}}{b_{k+1}}\ff + Q_k \left( \frac{1}{b_{k+1}}\ff s_k - \frac{1}{b_{k+1}}\ff e_{k} \frac{1}{b_k} + \frac{1}{b_k b_{k+1}} \right)\frac{Q_k}{b_{k+1}} \ff - \frac{1}{b_k^2b_{k+1}}\ff \\
=& Q_k \left( - \frac{1}{b_{k+1}}\ff e_{k} \frac{1}{b_k} + \frac{1}{b_k b_{k+1}} \right)\frac{Q_k}{b_{k+1}} \ff - \frac{1}{b_k^2b_{k+1}}\ff = - Q_k \frac{1}{b_{k+1}}\ff e_{k} \frac{Q_k}{b_k b_{k+1}} \ff\\
=& - A e_{k} \frac{1}{b_k b_{k+2}} A
\end{aligned}
\end{equation}

To summarize, we have:
\begin{equation} \label{eqn:rest1}
\begin{aligned}
& \text{(\ref{eqn:rhs3})}+\text{(\ref{eqn:rhs4})} + \text{(\ref{eqn:rhs7})}+\text{(\ref{eqn:rhs8})} -\text{(\ref{eqn:lhs3})}-\text{(\ref{eqn:lhs4})} + \text{(\ref{eqn:rhs6})}-\text{(\ref{eqn:lhs7})}-\text{(\ref{eqn:lhs8})}\\
=& A s_k (1-\ff) s_k \frac{b_{k+1}}{b_{k+2}}A + A \frac{1}{2\beta} e_k \frac{b_{k+1}}{b_k b_{k+2}} A - A e_k \frac{1}{b_k} (1-\ff) s_k \frac{b_{k+1}}{b_{k+2}}A \\
& - A e_{k} \frac{1}{b_k b_{k+2}} A
\end{aligned}
\end{equation}

We consider the following expressions
\begin{equation} \label{eqn:14b-13e-13f}
\begin{aligned}
\text{(\ref{eqn:rhs2})}-\text{(\ref{eqn:lhs5})}-\text{(\ref{eqn:lhs6})} =& A \left( s_k \highlightbox{\ff} s_{k+1} - b_{k+1} s_k \highlightbox{\ff} s_{k+1} \frac{1}{b_k} \highlightbox{\ff} + \frac{1}{b_k} \highlightbox{\ff} s_{k+1} \right) A \\
=& A \left( s_k s_{k+1} - \highlightbox{b_{k+1} s_k} s_{k+1} \frac{1}{b_k} + \frac{1}{b_k} s_{k+1} \right) A \\
=& A \left( s_k s_{k+1} - (s_k b_k - e_k + 1) s_{k+1} \frac{1}{b_k} + \frac{1}{b_k} s_{k+1} \right) A \\
=& A e_k s_{k+1} \frac{1}{b_k} A
\end{aligned}
\end{equation}
and
\begin{equation} \label{eqn:14e-13b}
\begin{aligned}
& \text{(\ref{eqn:rhs5})}-\text{(\ref{eqn:lhs2})} = Q_{k+1} s_{k+1} Q_{k+1} Q_k s_k Q_k \frac{1}{b_{k+1}} \ff - \frac{1}{b_k}\ff Q_{k+1} s_{k+1} Q_{k+1}Q_k s_k Q_k \\
=& A \left( s_{k+1} \highlightbox{\ff} s_k - s_{k+1} \frac{1}{b_k} \highlightbox{\ff} s_k b_{k+1}\right) A = A \left( s_{k+1} s_k - s_{k+1} \highlightbox{\frac{1}{b_k} s_k} b_{k+1}\right) A \\
=& A \left( s_{k+1} s_k - s_{k+1} \left( s_k \frac{1}{b_{k+1}} - \frac{1}{b_k}\ff e_k \frac{1}{b_{k+1}} + \frac{1}{b_kb_{k+1}} \right) b_{k+1}\right) A \\
=& A \left( s_{k+1} \frac{1}{b_k} \highlightbox{\ff} e_k - s_{k+1} \frac{1}{b_k} \right) A = A s_{k+1} \frac{1}{b_k} e_k A - A s_{k+1} \frac{1}{b_k} A.
\end{aligned}
\end{equation}

The only terms left are (\ref{eqn:lhs1}) and (\ref{eqn:rhs1}), which we need to expand into a large number of terms. We will do so separately, starting with the one from the left hand side.

\begin{align}
& \text{(\ref{eqn:rhs1})} = -A \left( s_{k+1} \ff s_{k} b_{k+2} \highlightbox{\ff} s_{k+1} \right) A \nonumber = -A \left( \frac{1}{b_{k+2}} \ff \highlightbox{b_{k+2} s_{k+1}} \ff s_{k} b_{k+2} s_{k+1} \right) A \nonumber\\
=& -A \frac{1}{b_{k+2}} \ff s_{k+1} b_{k+1} \highlightbox{\ff} s_{k} b_{k+2} s_{k+1} A + A \frac{1}{b_{k+2}} \ff e_{k+1} \highlightbox{\ff} s_{k} b_{k+2} s_{k+1} A - A s_{k} s_{k+1} A\nonumber\\
=& -A \frac{1}{b_{k+2}} \ff \highlightbox{s_{k+1} b_{k+1}} s_{k} b_{k+2} s_{k+1} A + A \frac{1}{b_{k+2}} \ff e_{k+1} s_{k} \highlightbox{b_{k+2} s_{k+1}} A - A s_{k} s_{k+1} A\nonumber\\
=& -A  s_{k+1} s_{k} \highlightbox{b_{k+2} s_{k+1}} A -A \frac{1}{b_{k+2}} \ff e_{k+1} s_{k} \highlightbox{b_{k+2} s_{k+1}} A +A s_{k} s_{k+1} A\nonumber\\
&+A \frac{1}{b_{k+2}} \ff e_{k+1} s_{k} s_{k+1} b_{k+1} A -A \frac{1}{b_{k+2}} \ff e_{k+1} s_{k} e_{k+1} A +A \frac{1}{b_{k+2}} \ff e_{k+1} s_{k} A\nonumber\\
&- A s_{k} s_{k+1} A\nonumber\\
=& -A  s_{k+1} s_{k} s_{k+1} b_{k+1} A +A  s_{k+1} s_{k} e_{k+1} A -A  s_{k+1} s_{k} A \nonumber\\
&-A \frac{1}{b_{k+2}} \ff e_{k+1} s_{k} s_{k+1} b_{k+1} A +A \frac{1}{b_{k+2}} \ff e_{k+1} s_{k} e_{k+1} A -A \frac{1}{b_{k+2}} \ff e_{k+1} s_{k} A\nonumber\\
& +A \frac{1}{b_{k+2}} \ff e_{k+1} s_{k} s_{k+1} b_{k+1} A -A \frac{1}{b_{k+2}} \ff e_{k+1} s_{k} e_{k+1} A +A \frac{1}{b_{k+2}} \ff e_{k+1} s_{k} A \nonumber\\
=& -A  s_{k+1} s_{k} s_{k+1} b_{k+1} A\label{eqn:14a1}\\
&+A  s_{k+1} s_{k} e_{k+1} A\label{eqn:14a2}\\
&-A  s_{k+1} s_{k} A\label{eqn:14a3}
\end{align}

For the other term we obain

\begin{align}
&-\text{(\ref{eqn:lhs1})} = A \left( b_{k+1} s_k \highlightbox{\ff} s_{k+1} \frac{1}{b_k}\ff s_k b_{k+1} \right) A \nonumber \\
=& A \left( \frac{b_{k+1}}{b_{k+2}} \ff s_k \highlightbox{b_{k+2} s_{k+1}} \frac{1}{b_k}\ff s_k b_{k+1} \right) A \nonumber\\
=& A \frac{b_{k+1}}{b_{k+2}}\ff \left( s_k \frac{1}{b_k} s_{k+1} b_{k+1} \highlightbox{\ff} s_k b_{k+1} - s_k \frac{1}{b_k} e_{k+1} \highlightbox{\ff} s_k b_{k+1} + s_k \frac{1}{b_k} \ff s_k b_{k+1}\right) A \nonumber\\
=& A \frac{b_{k+1}}{b_{k+2}}\ff s_k \frac{1}{b_k} s_{k+1} b_{k+1} s_k b_{k+1} A \label{eqn:13a1}\\
& - A \frac{b_{k+1}}{b_{k+2}}\ff s_k \frac{1}{b_k} e_{k+1} s_k b_{k+1}A\label{eqn:13a2}\\
& + A \frac{b_{k+1}}{b_{k+2}}\ff s_k \frac{1}{b_k} \ff s_k b_{k+1} A \label{eqn:13a3}.
\end{align}

We examine the terms (\ref{eqn:13a1}), (\ref{eqn:13a2}), (\ref{eqn:13a3}) separately.
\begin{align}
& \text{(\ref{eqn:13a1})} = A \frac{b_{k+1}}{b_{k+2}}\ff \highlightbox{s_k \frac{1}{b_k}} s_{k+1} b_{k+1} s_k b_{k+1} A \nonumber\\
=& A \left( \frac{1}{b_{k+2}}\ff s_k \highlightbox{s_{k+1} b_{k+1}} s_k - \frac{1}{b_{k+2}}\ff e_k \frac{1}{b_k} \highlightbox{s_{k+1} b_{k+1}} s_k + \frac{1}{b_{k+2}b_k}\ff \highlightbox{s_{k+1} b_{k+1}} s_k \right) b_{k+1} A \nonumber\\
=& A s_k s_{k+1} s_k b_{k+1} A \label{eqn:13a1-1}\\
&+ A \frac{1}{b_{k}}\ff s_k e_{k+1} s_k b_{k+1} A \label{eqn:13a1-2}\\
&- A \frac{b_{k+1}}{b_{k+2}} A \label{eqn:13a1-3}\\
& -A e_k \frac{1}{b_k} s_{k+1} s_k b_{k+1} A\label{eqn:13a1-4}\\
& -A \frac{1}{b_{k+2}} e_k \frac{1}{b_k} e_{k+1} s_k b_{k+1} A\label{eqn:13a1-5}\\
& +A \frac{1}{2\beta}\frac{1}{b_{k+2}} e_k \frac{b_{k+1}}{b_k} A\label{eqn:13a1-6}\\
& +A \frac{1}{b_k}\ff s_{k+1} s_k b_{k+1} A\label{eqn:13a1-7}\\
& +A \frac{1}{b_{k+2}b_k}\ff e_{k+1} s_k b_{k+1} A\label{eqn:13a1-8}\\
& -A \frac{1}{b_{k+2}b_k}\ff s_k b_{k+1} A\label{eqn:13a1-9}
\end{align}

\begin{align}
& \text{(\ref{eqn:13a2})} = - A \left( \frac{b_{k+1}}{b_{k+2}} \ff \highlightbox{s_k \frac{1}{b_k}} e_{k+1} s_k b_{k+1} \right) A \nonumber\\
=& - A \frac{1}{b_{k+2}} \ff s_k e_{k+1} s_k b_{k+1} A \label{eqn:13a2-1}\\
& + A \frac{1}{b_{k+2}} \ff e_k \frac{1}{b_k} e_{k+1} s_k b_{k+1} A \label{eqn:13a2-2}\\
& - A \frac{1}{b_k b_{k+2}} \ff e_{k+1} s_k b_{k+1}  A \label{eqn:13a2-3}
\end{align}

The sum of (\ref{eqn:13a3}) and (\ref{eqn:rest1}) simplifies to
\begin{align}
& \text{(\ref{eqn:13a3})} + \text{(\ref{eqn:rest1})} \nonumber \\
=& A \left( \frac{1}{b_{k+2}}\ff \highlightbox{b_{k+1}s_k} \frac{1}{b_k} \ff s_k b_{k+1}\right.  +s_k (1-\ff) s_k \frac{b_{k+1}}{b_{k+2}}A + A  \frac{1}{2\beta} e_k \frac{b_{k+1}}{b_k b_{k+2}} \nonumber\\
&- \left. e_k \frac{1}{b_k} (1-\ff) s_k \frac{b_{k+1}}{b_{k+2}} A - A e_{k} \frac{1}{b_k b_{k+2}} \right) A\nonumber\\
=& A \left( \frac{1}{b_{k+2}}\ff s_k \ff s_k b_{k+1} - \frac{1}{b_{k+2}}\ff e_k \frac{1}{b_k}\ff s_k b_{k+1} + \frac{1}{b_{k+2}b_k}\ff s_k b_{k+1}\right. \nonumber\\
&+s_k (1-\ff) s_k \frac{b_{k+1}}{b_{k+2}} + \frac{1}{2\beta} e_k \frac{b_{k+1}}{b_k b_{k+2}} - \left. e_k \frac{1}{b_k} (1-\ff) s_k \frac{b_{k+1}}{b_{k+2}} - e_{k} \frac{1}{b_k b_{k+2}} \right) A\nonumber\\
=&A \left( \frac{b_{k+1}}{b_{k+2}} - \highlightbox{e_k \frac{1}{b_k} s_k} \frac{b_{k+1}}{b_{k+2}} + \frac{1}{b_{k+2}b_k}\ff s_k b_{k+1} + \frac{1}{2\beta} e_k \frac{b_{k+1}}{b_k b_{k+2}} - e_{k} \frac{1}{b_k b_{k+2}}\right) A\nonumber\\
=&\frac{1}{b_{k}} \label{eqn:rest2-1}\\
& + A\frac{1}{b_k}\ff s_k \frac{b_{k+1}}{b_{k+2}} A \label{eqn:rest2-2}\\
& - A e_{k} \frac{1}{b_k b_{k+2}} A\label{eqn:rest2-3}.
\end{align}

We now compare all the results and see that the following summands cancel each other: (\ref{eqn:13a2-1}) and (\ref{eqn:13a1-2}), (\ref{eqn:13a2-2}) and (\ref{eqn:13a1-5}), (\ref{eqn:13a2-3}) and (\ref{eqn:13a1-8}), (\ref{eqn:rest2-1}) and (\ref{eqn:13a1-3}), (\ref{eqn:rest2-2}) and (\ref{eqn:13a1-9}), (\ref{eqn:14a1}) and (\ref{eqn:13a1-1}).

From the remaining summands we add the following:
\begin{equation}\label{eqn:finalrest}
\begin{aligned}
&\text{(\ref{eqn:13a1-4})} + \text{(\ref{eqn:14a2})} + \text{(\ref{eqn:14b-13e-13f})}\\
=&A \left(- e_k s_{k+1} \highlightbox{\frac{1}{b_k} s_k} b_{k+1} + \highlightbox{s_{k+1} s_{k} e_{k+1}} + e_k s_{k+1} \frac{1}{b_k} \right) A\\
=&A \left( - \highlightbox{e_k s_{k+1} s_k} + e_k s_{k+1} \frac{1}{b_k} e_k - e_k s_{k+1} \frac{1}{b_k}+ e_{k} e_{k+1} + e_k s_{k+1} \frac{1}{b_k} \right) A\\
=&A \left(- e_k e_{k+1} + e_k s_{k+1} \frac{1}{b_k} e_k + e_{k} e_{k+1}\right) A = A e_k s_{k+1} \frac{1}{b_k} e_k A\\
\end{aligned}
\end{equation}
A number of terms cancel each other:
\begin{equation}
\begin{aligned}
&\text{\eqref{eqn:13a1-7}} + \text{\eqref{eqn:14a3}} + \text{\eqref{eqn:14e-13b}}\\
=& A \left( s_{k+1} \highlightbox{\frac{1}{b_k} s_k} b_{k+1} - s_{k+1} s_{k} + s_{k+1} \frac{1}{b_k} e_k - s_{k+1} \frac{1}{b_k} \right) A\\
=& A \left( s_{k+1} s_k - s_{k+1} \frac{1}{b_k}e_k + s_{k+1} \frac{1}{b_k} - s_{k+1} s_{k} + s_{k+1} \frac{1}{b_k} e_k - s_{k+1} \frac{1}{b_k} \right) A =0
\end{aligned}
\end{equation}
Collecting all the remaining summands we get the following expression:
\begin{equation}
\begin{aligned}
&\text{\eqref{eqn:13a1-6}} + \text{\eqref{eqn:rest2-3}} + \text{\eqref{eqn:finalrest}}\\
=& A \left(\frac{1}{2\beta}\frac{1}{b_{k+2}} e_k \frac{b_{k+1}}{b_k} 
- e_{k} \frac{1}{b_k b_{k+2}}
+ \highlightbox{e_k s_{k+1}} \frac{1}{b_k} e_k \right) A\\
=& A \left( \frac{1}{2\beta}\frac{1}{b_{k+2}} e_k \frac{b_{k+1}}{b_k} - e_{k} \frac{1}{b_k b_{k+2}} + e_k e_{k+1} \highlightbox{\ff s_{k} \frac{1}{b_k} e_k} \right) A \\
=&A \left( \frac{1}{2\beta}\frac{1}{b_{k+2}} e_k \frac{b_{k+1}}{b_k} - e_{k} \frac{1}{b_k b_{k+2}} + \frac{1}{2\beta}\frac{1}{b_{k+2}} e_k e_{k+1} \highlightbox{\ff} e_k \right) A \\
=& A \left(e_k \frac{b_{k+1} - 2 \beta + b_k}{2\beta b_{k+2}b_k} \right)A = A \left(e_k \frac{y_{k+1} + y_k}{2\beta b_{k+2}b_k} \right)A = 0.
\end{aligned}
\end{equation}
This proves the claim of the proposition.
\end{proof}

\subsection{Relations involving only $\tilde{e}_k$'s in $\ff \VWdcycl \ff$}

We continue with the defining relations that will only involve the $\tilde{e}_k$'s. The key relation that $\tilde{e}_k$ squares to $\delta \tilde{e}_k$ was already proven in Lemma \ref{lem:e2=deltae}. We continue with the remaining relations:

\begin{lemma} \label{lem:eiej=ejei}
We have $\tilde{e}_i \tilde{e}_j= \tilde{e}_j \tilde{e}_i$ for $1 \leq i,j < d$ with $|i-j| > 1$.
\end{lemma}
\begin{proof}
Since $\ff e_i \ff$ commutes with $Q_j$ when $| i - j | > 1$ the statement follows.
\end{proof}

\begin{lemma} \label{lem:eee=e}
We have $\tilde{e}_k \tilde{e}_{k+1} \tilde{e}_k= \tilde{e}_{k}$ and $\tilde{e}_{k+1} \tilde{e}_{k} \tilde{e}_{k+1}= \tilde{e}_{k+1}$ for $1 \leq k < d-1$.
\end{lemma}
\begin{proof}
We only prove the first equality, the second is done in an analogous way. We compute
\begin{eqnarray*}
\tilde{e}_k \tilde{e}_{k+1} \tilde{e}_k &=& Q_k e_k Q_k Q_{k+1}e_{k+1} Q_{k+1} Q_k e_k Q_k\\
&=&  Q_k \sqrt{b_{k+2}}\ff e_k \ff e_{k+1} \frac{1}{b_k} \ff e_k \sqrt{b_{k+2}} \ff Q_k \\
&\overset{(a)}{=}& Q_k \sqrt{b_{k+2}}\ff e_k \ff e_{k+1} \frac{1}{c_{k+1}} \ff e_k \sqrt{b_{k+2}} \ff Q_k\\
&\overset{(b)}{=}& Q_k \sqrt{b_{k+2}}\ff e_k \ff e_{k+1} \frac{1}{b_{k+2}} \ff e_k \sqrt{b_{k+2}} \ff Q_k\\
&=& Q_k \sqrt{b_{k+2}}\ff e_k \ff e_{k+1} \ff e_k \frac{\sqrt{b_{k+2}}}{b_{k+2}} \ff Q_k \overset{(c)}{=} \tilde{e}_k
\end{eqnarray*}
Where $(a)$ is due to Relation~(VW.\ref{8b}), since $\frac{1}{b_k} \ff e_k = \frac{1}{c_{k+1}} \ff e_k$, and $(b)$ follows from $e_{k+1} \frac{1}{c_{k+1}} \ff = e_{k+1}\frac{1}{b_{k+2}} \ff$ due to Relation~(VW.\ref{8a}). Finally $(c)$ is again a consequence of Lemma \ref{lem:commute_b} and Relation~(VW.\ref{6d}).
\end{proof}

\subsection{Mixed relations in $\ff \VWdcycl \ff$}

We are now left with proving the relations involving both $\tilde{s}_j$'s and $\tilde{e}_j$'s.

\begin{lemma} \label{lem:es=e}
We have $\tilde{e}_{k} \tilde{s}_{k} = \tilde{e}_{k}= \tilde{s}_{k} \tilde{e}_{k}$ for $1 \leq k < d$.
\end{lemma}
\begin{proof}
We compute
\begin{eqnarray*}
\tilde{e}_k \tilde{s}_k &=& -Q_k e_k \frac{b_{k+1}}{b_k} \ff s_k Q_k + Q_k e_k Q_k \frac{1}{b_k}\ff \\
&\overset{(a)}{=}& -Q_k e_k \frac{b_{k+1}}{b_k} s_k Q_k + Q_k e_k Q_k \frac{1}{b_k}\ff \\
&\overset{(b)}{=}& -Q_k e_k \frac{1}{b_k} s_k b_{k}Q_k + Q_k e_k \frac{1}{b_k} e_k Q_k\\
&\overset{(c)}{=}& - \frac{1}{2\beta} Q_k e_k Q_k + (1+\frac{1}{2\beta}) Q_k e_k Q_k = \tilde{e}_k
\end{eqnarray*}
where equality $(a)$ is due to Proposition \ref{prop:kill_f}, $(b)$ is due to Relation~(VW.\ref{7}), and finally $(c)$ is due to Proposition \ref{prop:reduce_fracs}. The second equality in the claim follows analogously.
\end{proof}

\begin{lemma} \label{lem:se=es}
We have $\tilde{s}_i \tilde{e}_j= \tilde{e}_j \tilde{s}_i$ for $1 \leq i,j < d$ with $|i-j| > 1$.
\end{lemma}
\begin{proof}
This follows by the same arguments as for Lemmas \ref{lem:sisj=sjsi} and \ref{lem:eiej=ejei}.
\end{proof}

\begin{lemma} \label{lem:see=se}
We have $\tilde{s}_k \tilde{e}_{k+1} \tilde{e}_k= \tilde{s}_{k+1}\tilde{e}_{k}$ and $\tilde{s}_{k+1} \tilde{e}_{k} \tilde{e}_{k+1}= \tilde{s}_{k}\tilde{e}_{k+1}$ for $1 \leq k < d-1$.
\end{lemma}
\begin{proof}
We only prove the first equality, the second is done analogously,
\begin{equation*}
\begin{aligned}
\tilde{s}_k \tilde{e}_{k+1} \tilde{e}_k &= -Q_k s_k Q_k Q_{k+1}e_{k+1} Q_{k+1} Q_k e_k Q_k + \frac{1}{b_{k}}\ff  Q_{k+1} e_{k+1} Q_{k+1} Q_k e_k Q_k \\
&\overset{(a)}{=} - Q_k \sqrt{b_{k+2}}\ff s_k \frac{1}{b_k} e_{k+1} e_k \ff \sqrt{b_{k+2}} \ff Q_k + \frac{\sqrt{b_{k+2}}}{b_k\sqrt{b_{k}}\sqrt{b_{k+1}}}\ff e_{k+1} e_k \ff \sqrt{b_{k+2}} \ff Q_k\\
&\overset{(b)}{=} - Q_{k+1} \frac{1}{\sqrt{b_{k}}}\ff s_k e_{k+1} e_k \ff \sqrt{b_{k+2}} \ff Q_k + Q_{k+1} \frac{1}{\sqrt{b_{k}}}\ff e_{k} \frac{1}{b_k} e_{k+1} e_k \ff \sqrt{b_{k+2}} \ff Q_k\\
&\overset{(c)}{=} - Q_{k+1} \frac{1}{\sqrt{b_{k}}}\ff s_{k+1} e_k \ff \sqrt{b_{k+2}} \ff Q_k + \frac{1}{b_{k+2}} Q_{k+1} \frac{1}{\sqrt{b_{k}}}\ff e_{k} \ff \sqrt{b_{k+2}} \ff Q_k\\
&\overset{(d)}{=} - Q_{k+1} s_{k+1} Q_{k+1} Q_k e_k Q_k + \frac{1}{b_{k+1}} Q_{k} e_{k} Q_k = \tilde{s}_{k+1} \tilde{e}_k.
\end{aligned}
\end{equation*}
Where equality $(a)$ follows from Proposition \ref{prop:kill_f}, while the second equality is a consequence of Lemma \ref{lem:commute_b} and again Proposition \ref{prop:kill_f}. Equality $(c)$ follows by using relations (\ref{8a}) and (\ref{8b}) to rewrite the second summand and then applying relations (\ref{6b}) and (\ref{6d}). Finally equality $(d)$ is using Proposition \ref{prop:kill_f} and reordering the factors afterwards.
\end{proof}

\section{Example: the graded Brauer algebras ${\rm Br}^{\rm gr}_2(\delta)$} \label{sec:example}

In this section we will illustrate explicitly the construction of the isomorphism for the Brauer algebras ${\rm Br}_2(\delta)$ and describe their graded version ${\rm Br}^{\rm gr}_2(\delta)$.

\subsection{Case $\delta\neq 0$}
We first consider the case ${\rm Br}_2(\delta)$ for $\delta \neq 0$. By \cite{Rui} this Brauer algebra is semisimple with basis $1$, $t=t_1$, and $g=g_1$.  The set of orthogonal idempotents is 
\begin{eqnarray*}
\left\{\frac{1+t}{2}-\frac{1}{\delta}g,\, \frac{1-t}{2}+\frac{1}{\delta}g,\, \frac{1}{\delta}g \right\}
\end{eqnarray*}
which gives rise to an isomorphism
$$ {\rm Br}_2(\delta) \cong \mathbb{C} \oplus \mathbb{C} \oplus \mathbb{C}.$$
Note that the grading on ${\rm Br}_2(\delta)$ needs to be trivial since all idempotents have to have degree $0$. We now want to illustrate the idempotent truncation of the level $2$ cyclotomic quotient $\VW_2(\Xi)$ with parameters from Definition\ref{defalphabeta}. 

We describe  $\VW_2^{\rm cycl}$ in terms of the seminormal representation of $\VW_2^{\rm cycl}$ from \cite[Theorem 4.13]{AMR} by an action of $\VW_2^{\rm cycl}$ on the vector space with basis given by all up-down bitableaux of length $2$. Explicitly, this basis consists of
$$
\begin{xy}
  \xymatrix@R=4pt@C=2pt{
v_1 \ar@{=}[d] & v_2 \ar@{=}[d] & v_3 \ar@{=}[d] & v_4 \ar@{=}[d] & v_5 \ar@{=}[d] & v_6 \ar@{=}[d] & v_7 \ar@{=}[d] & v_8 \ar@{=}[d] \\
\left( \emptyset , \emptyset \right) \ar@{-}[d] & \left( \emptyset , \emptyset \right) \ar@{-}[d] &\left( \emptyset , \emptyset \right) \ar@{-}[d] &\left( \emptyset , \emptyset \right) \ar@{-}[d] &\left( \emptyset , \emptyset \right) \ar@{-}[d] &\left( \emptyset , \emptyset \right) \ar@{-}[d] &\left( \emptyset , \emptyset \right) \ar@{-}[d] &\left( \emptyset , \emptyset \right) \ar@{-}[d]\\
\left( \yng(1), \emptyset \right) \ar@{-}[d] & \left( \yng(1), \emptyset \right) \ar@{-}[d] & \left( \yng(1), \emptyset \right) \ar@{-}[d] & \left( \emptyset, \yng(1) \right) \ar@{-}[d] & \left( \yng(1), \emptyset \right) \ar@{-}[d] & \left( \emptyset, \yng(1) \right) \ar@{-}[d] & \left( \emptyset, \yng(1) \right) \ar@{-}[d] & \left( \emptyset, \yng(1) \right) \ar@{-}[d]\\
\left( \yng(2) , \emptyset \right) & \left( \yng(1,1) , \emptyset \right) & \left( \emptyset , \emptyset \right) & \left( \emptyset , \emptyset \right) & \left( \yng(1) , \yng(1) \right) & \left( \yng(1) , \yng(1) \right) &\left( \emptyset , \yng(1,1) \right) & \left( \emptyset , \yng(2) \right)
}
\end{xy}
$$
all of which are common eigenvectors for $y_1$ and $y_2$. The corresponding pairs of eigenvalues are the following:
$$
\begin{array}{cccccccc}
v_1 & v_2 & v_3 & v_4 & v_5 & v_6 & v_7 & v_8\\
(\alpha,\alpha-1) & (\alpha, \alpha+1) & (\alpha,-\alpha) & (\beta,-\beta) & (\alpha,\beta) & (\beta, \alpha) & (\beta,\beta-1) & (\beta,\beta+1) 
\end{array}
$$
where the first entry denotes the eigenvalue for $y_1$ and the second for $y_2$. Using these eigenvalues one can calculate via  \cite[Theorem 4.13]{AMR} the matrix of $s_1$ 
$$
s_1=\left(\begin{array}{cccccccc}
-1 & 0 & 0 & 0 & 0 & 0 & 0 & 0\\
0 & 1 & 0 & 0 & 0 & 0 & 0 & 0\\
0 & 0 & \frac{\alpha+\beta-1}{\alpha-\beta} & \frac{\sqrt{-(2\beta-1)(2\alpha-1)}}{\alpha-\beta} & 0 & 0 & 0 & 0\\
0 & 0 & \frac{\sqrt{-(2\beta-1)(2\alpha-1)}}{\alpha-\beta} & -\frac{\alpha+\beta-1}{\alpha-\beta} & 0 & 0 & 0 & 0\\
0 & 0 & 0 & 0 & \star & \star & 0 & 0\\
0 & 0 & 0 & 0 & \star & \star & 0 & 0\\
0 & 0 & 0 & 0 & 0 & 0 & \star & 0\\
0 & 0 & 0 & 0 & 0 & 0 & 0 & \star
\end{array} \right).
$$
The $\star$'s indicate some further non-zero entries which are irrelevant for the construction. Similarly one obtains the matrix for $e_1$ as
$$
e_1=\left(\begin{array}{cccccccc}
0 & 0 & 0 & 0 & 0 & 0 & 0 & 0\\
0 & 0 & 0 & 0 & 0 & 0 & 0 & 0\\
0 & 0 & (2\alpha-1)\frac{\alpha+\beta}{\alpha-\beta} & \sqrt{-(2\beta-1)(2\alpha-1)}\frac{\alpha+\beta}{\alpha-\beta} & 0 & 0 & 0 & 0\\
0 & 0 & \sqrt{-(2\beta-1)(2\alpha-1)}\frac{\alpha+\beta}{\alpha-\beta} & -(2\beta-1)\frac{\alpha+\beta}{\alpha-\beta} & 0 & 0 & 0 & 0\\
0 & 0 & 0 & 0 & 0 & 0 & 0 & 0\\
0 & 0 & 0 & 0 & 0 & 0 & 0 & 0\\
0 & 0 & 0 & 0 & 0 & 0 & 0 & 0\\
0 & 0 & 0 & 0 & 0 & 0 & 0 & 0
\end{array} \right).
$$
Although tedious, it is of course straightforward to check that these matrices together with the action of $y_1,y_2$ satisfy all the defining relations of $\VW_2^{\rm cycl}$.

 For the isomorphism in Theorem \ref{thm:main} we first need to apply the idempotent $\ff$ from both sides. In the chosen basis this amounts to truncation with respect to the basis vectors where the second partition is always empty (i.e. $v_1$, $v_2$ and $v_3$) that means we look at the submatrices consisting of the first three columns in the first three rows. Clearly, the  submatrices $\ff s_1 \ff$ and $\ff e_1 \ff$ do not even satisfy the most basic Brauer algebra relations, i.e. the relation for squares. This deficiency is overcome by the correction term $Q=Q_1=\sqrt{\frac{b_2}{b_1}}\ff$ and $\frac{1}{b_1}\ff$ from the definition of the isomorphism in Theorem~\ref{thm:main} as we show now explicitly.  We have
$$
Q=\left(\begin{array}{ccc}
 \sqrt{\frac{\alpha+\beta-1}{\alpha+\beta}}& 0 & 0\\
 0 & \sqrt{\frac{\alpha+\beta+1}{\alpha+\beta}} & 0\\
 0 & 0 &  \sqrt{\frac{\beta-\alpha}{\alpha+\beta}}
\end{array}\right) \text{ and }
\frac{1}{b_1}\ff=\left(\begin{array}{ccc}
 \frac{1}{\alpha+\beta}& 0 & 0\\
 0 & \frac{1}{\alpha+\beta} & 0\\
 0 & 0 &  \frac{1}{\alpha+\beta}
\end{array}\right).
$$
By multiplying $\ff e_1 \ff$ from both sides with $Q$ we obtain
$$
Q \ff e_1 \ff Q =\left(\begin{array}{ccc}
 0 & 0 & 0\\
 0 & 0 & 0\\
 0 & 0 & \delta
\end{array}\right),
$$
which is obviously a matrix that squares to $\delta$ times itself. The analogous construction for  $\ff s_1 \ff$ needs an extra correction term (as given in Theorem~\ref{thm:main}): 
$$
Q \ff s_1 \ff Q =\left(\begin{array}{ccc}
 -\frac{\alpha+\beta-1}{\alpha+\beta} & 0 & 0\\
 0 & \frac{\alpha+\beta+1}{\alpha+\beta} & 0\\
 0 & 0 & -\frac{\alpha+\beta-1}{\alpha+\beta}
\end{array}\right) \text{ and } -Q \ff s_1 \ff Q + \frac{1}{b_1}\ff  = \left(\begin{array}{ccc}
1 & 0 & 0\\
 0 & -1 & 0\\
 0 & 0 & 1
\end{array}\right).
$$

\subsection{Case $\delta\neq 0$}If we try to do the same for the Brauer algebra ${\rm Br}_2(0)$ we immediately encounter a problem, since the algebra is not semisimple and so we cannot apply the formulas and constructions from \cite{AMR}. The whole difficulty of our proof is to show that the formulas in Theorem~\ref{thm:main} still make sense and give the required correction terms even in the non-semisimple case.

One can show that via the orthogonal idempotents $\{\frac{1+t}{2},\frac{1-t}{2}\}$ one obtains
\begin{eqnarray}
\label{Schluss}
{\rm Br}_2(0) &\cong&\mathbb{C} \oplus \mathbb{C}[x]/(x^2), 
\end{eqnarray}
with the element $x$ corresponding to the element $g \in {\rm Br}_2(0)$  It becomes graded in the obvious way by putting the idempotents in degree $0$ and $x$ in degree $2$. 

There is in fact a generalization of the up-down tableaux basis in the graded setting via the diagram calculus developed in \cite{ES3}. The explicit isomorphism between this description and the Brauer algebra itself is being worked out in \cite{LiS}.

\begin{remark}
To make the connection from \eqref{Schluss} to \cite{ES3} we note that ${\rm Br}^{\rm gr}_2(0)$ can be realized as the subalgebra of the generalized type ${\rm D}$ Khovanov algebra from \cite{ES2} (using the notation from there) with the following diagrams as basis
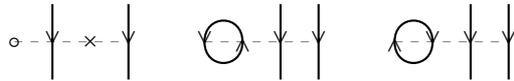
\begin{equation*}
\begin{tikzpicture}[thick,scale=0.5]
\begin{scope}[xshift=0cm]
\draw[thin,gray,dashed] (0,0) -- +(3,0);
\node at (0,-.04) {$\mathbf{\circ}$};
\node at (1,.1) {$\mathbf{\down}$};
\draw (1,1) -- +(0,-1);
\node at (2,0) {$\mathbf{\times}$};
\node at (3,.1) {$\mathbf{\down}$};
\draw (3,1) -- +(0,-1);
\draw (1,0.1) -- +(0,-1.1);
\draw (3,0.1) -- +(0,-1.1);
\end{scope}

\begin{scope}[xshift=5cm]
\draw[thin,gray,dashed] (0,0) -- +(3,0);
\node at (0,.1) {$\mathbf{\down}$};
\node at (1,-.1) {$\mathbf{\up}$};
\draw (0,0) .. controls +(0,.75) and +(0,.75) .. +(1,0);
\draw (0,0) .. controls +(0,-.75) and +(0,-.75) .. +(1,0);
\draw (2,1) -- +(0,-1);
\node at (2,.1) {$\mathbf{\down}$};
\draw (3,1) -- +(0,-1);
\node at (3,.1) {$\mathbf{\down}$};
\draw (2,0.1) -- +(0,-1.1);
\draw (3,0.1) -- +(0,-1.1);
\end{scope}

\begin{scope}[xshift=10cm]
\draw[thin,gray,dashed] (0,0) -- +(3,0);
\node at (0,-.1) {$\mathbf{\up}$};
\node at (1,.1) {$\mathbf{\down}$};
\draw (0,0) .. controls +(0,.75) and +(0,.75) .. +(1,0);
\draw (0,0) .. controls +(0,-.75) and +(0,-.75) .. +(1,0);
\draw (2,1) -- +(0,-1);
\node at (2,.1) {$\mathbf{\down}$};
\draw (3,1) -- +(0,-1);
\node at (3,.1) {$\mathbf{\down}$};
\draw (2,0.1) -- +(0,-1.1);
\draw (3,0.1) -- +(0,-1.1);
\end{scope}
\end{tikzpicture}
\end{equation*}
Here the first diagram spans the copy of $\mathbb{C}$ iunder the identification with \eqref{Schluss}, while the other two span a copy of $\mathbb{C}[x]/(x^2)$, with the second diagram being the unit and the third one being  the element of degree $2$, i.e. it corresponds to $x$.
\end{remark}

\bibliographystyle{alpha}
\bibliography{references}

\begin{thebibliography}{CDVM09}

\bibitem[AMR06]{AMR}
S.~Ariki, A.~Mathas, and H.~Rui.
\newblock Cyclotomic {N}azarov-{W}enzl algebras.
\newblock {\em Nagoya Math. J.}, 182:47--134, 2006.

\bibitem[Bac99]{Backelin}
E.~Backelin.
\newblock Koszul duality for parabolic and singular category {$\mathcal{O}$}.
\newblock {\em Represent. Theory}, 3:139--152 (electronic), 1999.

\bibitem[BGS96]{BGS}
A.~Beilinson, V.~Ginzburg, and W.~Soergel.
\newblock Koszul duality patterns in representation theory.
\newblock {\em J. Amer. Math. Soc.}, 9(2):473--527, 1996.

\bibitem[BK08]{BK}
J.~Brundan and A.~Kleshchev.
\newblock Schur-{W}eyl duality for higher levels.
\newblock {\em Selecta Math. (N.S.)}, 14(1):1--57, 2008.

\bibitem[Boe88]{Boe}
B.~D. Boe.
\newblock Kazhdan-{L}usztig polynomials for {H}ermitian symmetric spaces.
\newblock {\em Trans. Amer. Math. Soc.}, 309(1):279--294, 1988.

\bibitem[Bra37]{Brauer}
R.~Brauer.
\newblock On algebras which are connected with the semisimple continuous
  groups.
\newblock {\em Ann. of Math. (2)}, 38(4):857--872, 1937.

\bibitem[BS11]{BS3}
J.~Brundan and C.~Stroppel.
\newblock Highest weight categories arising from {K}hovanov's diagram algebra
  {III}: category {$\mathcal{O}$}.
\newblock {\em Represent. Theory}, 15:170--243, 2011.

\bibitem[CDVM09]{CDM}
A.~Cox, M.~De~Visscher, and P.~Martin.
\newblock The blocks of the {B}rauer algebra in characteristic zero.
\newblock {\em Represent. Theory}, 13:272--308, 2009.

\bibitem[dCP76]{CP}
C.~de~Concini and C.~Procesi.
\newblock A characteristic free approach to invariant theory.
\newblock {\em Advances in Math.}, 21(3):330--354, 1976.

\bibitem[Don98]{Donkin}
S.~Donkin.
\newblock {\em The {$q$}-{S}chur algebra}, volume 253 of {\em London
  Mathematical Society Lecture Note Series}.
\newblock Cambridge University Press, Cambridge, 1998.

\bibitem[DRV13]{DRV}
Z.~Daugherty, A.~Ram, and R.~Virk.
\newblock Affine and degenerate affine {BMW} algebras: actions on tensor space.
\newblock {\em Selecta Math. (N.S.)}, 19(2):611--653, 2013.

\bibitem[ES13a]{ES2}
M.~{Ehrig} and C.~{Stroppel}.
\newblock {Diagrams for perverse sheaves on isotropic Grassmannians and the
  supergroup SOSP(m,2n)}.
\newblock {\em ArXiv e-prints}, June 2013.

\bibitem[ES13b]{ES3}
M.~{Ehrig} and C.~{Stroppel}.
\newblock {Nazarov-Wenzl algebras, coideal subalgebras and categorified skew
  Howe duality}.
\newblock {\em ArXiv e-prints}, October 2013.

\bibitem[ES14]{ESSchurWeyl}
M.~{Ehrig} and C.~{Stroppel}.
\newblock {Schur-Weyl duality for the Brauer algebra and the ortho-symplectic
  Lie superalgebra}.
\newblock {\em ArXiv e-prints}, December 2014.

\bibitem[GL96]{GrahamLehrer}
J.~J. Graham and G.~I. Lehrer.
\newblock Cellular algebras.
\newblock {\em Invent. Math.}, 123(1):1--34, 1996.

\bibitem[GW09]{GW}
R.~Goodman and N.~R. Wallach.
\newblock {\em Symmetry, representations, and invariants}, volume 255 of {\em
  Graduate Texts in Mathematics}.
\newblock Springer, Dordrecht, 2009.

\bibitem[HM13]{HM}
J.~{Hu} and A.~{Mathas}.
\newblock {Seminormal forms and cyclotomic quiver Hecke algebras of type $A$}.
\newblock {\em ArXiv e-prints}, April 2013.

\bibitem[Hum08]{Humphreys}
J.~E. Humphreys.
\newblock {\em Representations of semisimple {L}ie algebras in the {BGG}
  category {$\mathcal{O}$}}, volume~94 of {\em Graduate Studies in
  Mathematics}.
\newblock American Mathematical Society, Providence, RI, 2008.

\bibitem[KX98]{Koenig-Xi}
S.~K{\"o}nig and C.~Xi.
\newblock On the structure of cellular algebras.
\newblock In {\em Algebras and modules, {II} ({G}eiranger, 1996)}, volume~24 of
  {\em CMS Conf. Proc.}, pages 365--386. Amer. Math. Soc., Providence, RI,
  1998.

\bibitem[{Li}14]{Li}
G.~{Li}.
\newblock {A KLR Grading of the Brauer Algebras}.
\newblock {\em ArXiv e-prints}, September 2014.

\bibitem[LS12]{LS}
T.~Lejczyk and C.~Stroppel.
\newblock A graphical description of $({D}_n,{A}_{n-1})$ {K}azhdan-{L}usztig
  polynomials.
\newblock {\em GMJ}, 55:313--340, 2012.

\bibitem[LS15]{LiS}
G.~Li and C~Stroppel.
\newblock An isomorphism of graded {B}rauer algebras, 2015.
\newblock in preparation.

\bibitem[LZ14]{LZ}
G.~Lehrer and R.~Zhang.
\newblock The first fundamental theorem of invariant theory for the
  orthosymplectic supergroup, 2014.
\newblock arXiv:1401.7395.

\bibitem[Naz96]{Nazarov}
M.~Nazarov.
\newblock Young's orthogonal form for {B}rauer's centralizer algebra.
\newblock {\em J. Algebra}, 182(3):664--693, 1996.

\bibitem[Rui05]{Rui}
H.~Rui.
\newblock A criterion on the semisimple {B}rauer algebras.
\newblock {\em J. Combin. Theory Ser. A}, 111(1):78--88, 2005.

\bibitem[SS14]{SS2}
A.~{Sartori} and C.~{Stroppel}.
\newblock {Walled Brauer algebras as idempotent truncations of level 2
  cyclotomic quotients}.
\newblock {\em ArXiv e-prints}, November 2014.

\bibitem[Str03]{Stroppel}
C.~Stroppel.
\newblock Category {$\mathcal{O}$}: gradings and translation functors.
\newblock {\em J. Algebra}, 268(1):301--326, 2003.

\end{thebibliography}

\end{document}